\documentclass[11pt]{amsart}

\usepackage{mathpazo}
\usepackage{amsmath,amssymb,amscd,color}
\usepackage{mathrsfs}
\usepackage{fullpage}
\usepackage{float}
\usepackage{tikz}
\usetikzlibrary{positioning}
\usetikzlibrary{decorations.markings}
\usetikzlibrary{arrows, scopes}
\usepackage{comment}
\definecolor{green}{RGB}{0,144,0}
\definecolor{bluegreen}{RGB}{17,100,180}
\usepackage[linkcolor = bluegreen, citecolor = bluegreen, colorlinks = True]{hyperref}
\numberwithin{equation}{section}
\numberwithin{figure}{section}
\numberwithin{table}{section}
\usepackage[all]{hypcap}
\usepackage{mathtools}
\newcommand\blfootnote[1]{%
  \begingroup
  \renewcommand\thefootnote{}\footnote{#1}%
  \addtocounter{footnote}{-1}%
  \endgroup
}

\newtheorem{theorem}{Theorem}[section]
\newtheorem*{theorem*}{Theorem}

\newtheorem{lemma}[theorem]{Lemma}

\newtheorem{proposition}[theorem]{Proposition}

\newcommand{\teichmuller}{Teichm{\"u}ller{ }}

\newcommand{\calH}{\mathcal{H}}
\newcommand{\hyp}{\mathcal{H}^{hyp}}
\newcommand{\odd}{\mathcal{H}^{odd}}
\newcommand{\even}{\mathcal{H}^{even}}
\newcommand{\non}{\mathcal{H}^{nonhyp}}
\newcommand{\T}{\mathcal{T}}
\newcommand{\C}{\mathcal{C}}
\newcommand{\D}{\mathcal{D}}
\newcommand{\Z}{\mathbb{Z}}

\newcommand{\Cbb}{\mathbb{C}}
\newcommand{\R}{\mathbb{R}}
\newcommand{\Hbb}{\mathbb{H}}
\newcommand{\sys}{\text{sys}}
\newcommand{\Mod}{\text{Mod}}
\newcommand{\SL}{\text{SL}}
\newcommand{\SO}{\text{SO}}
\newcommand{\GL}{\text{GL}}
\newcommand{\ind}{\text{ind}}
\newcommand{\bs}{\boldsymbol}

\begin{document}

\title[Single-cylinder square-tiled surfaces and ratio-optimising pseudo-Anosovs]{Single-cylinder square-tiled surfaces and the ubiquity of ratio-optimising pseudo-Anosovs}

\author[Jeffreys]{Luke Jeffreys}
\address{\hskip-\parindent
     School of Mathematics \\
     University of Bristol \\
     Fry Building \\
	Woodland Road \\
	Bristol BS8 1UG }
\email{luke.jeffreys@bristol.ac.uk}


\begin{abstract}
In every connected component of every stratum of Abelian differentials, we construct square-tiled surfaces with one vertical and one horizontal cylinder. We show that for all but the hyperelliptic components this can be achieved in the minimum number of squares necessary for a square-tiled surface in that stratum. For the hyperelliptic components, we show that the number of squares required is strictly greater and construct surfaces realising these bounds. Using these surfaces, we demonstrate that pseudo-Anosov homeomorphisms optimising the ratio of Teichm{\"u}ller to curve graph translation length are, in a reasonable sense, ubiquitous in the connected components of strata of Abelian differentials. Finally, we present a further application to filling pairs on punctured surfaces by constructing filling pairs whose algebraic and geometric intersection numbers are equal.
\end{abstract}


\maketitle


\section{Introduction}

\blfootnote{MSC Classification - Primary: 32G15, 30F30, 30F60. Secondary: 57M50. Keywords: Abelian differentials, square-tiled surfaces, pseudo-Anosov homeomorphisms}
Let $S$ be a closed, connected, oriented surface of genus $g\geq{2}$, and let $\calH$ be the moduli space of Abelian differentials on $S$; that is, the moduli space of pairs $(S,\omega)$ where $S$ is a closed, connected, Riemann surface of genus $g$ and $\omega$ is a non-zero holomorphic 1-form on $S$. It can be seen that $\calH$ is a complex algebraic orbifold of dimension $4g-3$. We stratify $\calH$ by the orders of the zeros of the Abelian differential. That is, the stratum $\calH(k_{1},\ldots,k_{n})$, with $k_{i}\geq{1}$ and $\Sigma_{i=1}^{n}k_{i}=2g-2$, is the subset of $\calH$ consisting of Abelian differentials with $n$ zeros of orders $k_{1},\ldots,k_{n}$. Each stratum $\calH(k_{1},\ldots,k_{n})$ is a complex algebraic orbifold of dimension $2g+n-1$.

Masur \cite{Ma} and Veech \cite{V1} proved independently that the \teichmuller geodesic flow acts ergodically on each connected component of each stratum of the moduli space of unit area quadratic differentials on a surface. For the strata of quadratic differentials that are squares of Abelian differentials, these ergodic components are given by the connected components of the strata of the moduli space of Abelian differentials. Kontsevich-Zorich determined the number of connected components of each stratum of Abelian differentials showing, using the properties of hyperellipticity and spin parity, that a stratum can have up to three connected components {\cite[Theorems 1 and 2]{KZ}}. 

A connected component is said to be hyperelliptic if it consists entirely of hyperelliptic Abelian differentials. Similarly, a connected component is said to be even, respectively odd, if it consists entirely of Abelian differentials with even, respectively odd, spin parity. The maximum of three connected components is realised by strata having a hyperelliptic component, an even component, and an odd component. The hyperelliptic components in the strata $\calH(2g-2)$ and $\calH(g-1,g-1)$ are denoted by $\hyp(2g-2)$ and $\hyp(g-1,g-1)$, respectively. In fact, these are the only strata that have hyperelliptic components.

\subsection*{Square-tiled surfaces} A {\em square-tiled surface} in a stratum of $\calH$ is an Abelian differential given by a branched cover of the square torus, and one can think of square-tiled surfaces as being the integral points of the period coordinates on that stratum. As such, understanding square-tiled surfaces has played a crucial role in calculating the volumes of strata of Abelian differentials. See, for example, the works of Zorich \cite{Z2} and Eskin-Okounkov \cite{EO}. Such calculations depend on asymptotic counts of square-tiled surfaces. While Eskin-Okounkov performed the count globally using the representation theory of the symmetric group, Zorich separated the count of the square-tiled surfaces according to their combinatorial type.

Given a square-tiled surface, one important piece of combinatorial data is the number of maximally embedded annuli in the horizontal and vertical directions, respectively called {\em horizontal} and {\em vertical cylinders}. It is a consequence of recent work of Delecroix-Goujard-Zograf-Zorich that square-tiled surfaces with one vertical and one horizontal cylinder, which we shall call {\em 1,1-square-tiled surfaces}, make a non-zero contribution to the volumes of strata of Abelian differentials {\cite[Section 2]{DGZZ}}, and moreover equidistribute as the number of squares tends to infinity. Indeed, they showed that this equidistribution is true more generally for square-tiled surfaces of fixed combinatorial type in any $\GL(2,\R)$-invariant suborbifold containing a single square-tiled surface {\cite[Theorem 1.4]{DGZZ}}. By this we mean, in any finite volume open subset $U$, a point chosen at random from an $\epsilon$-grid in $U$ is a square-tiled surface having the desired combinatorics with probability that, as $\epsilon$ tends to zero, does not depend on $U$.

An Euler-characteristic argument shows that the minimum number of squares required for a square-tiled surface in the stratum $\calH(k_{1},\ldots,k_{n})$ is $2g+n-2$. Square-tiled surfaces realising this number exist in every connected component. Indeed, one can take square-tiled surfaces constructed from the Jenkins-Strebel permutation representatives given by Zorich~\cite{Z}. However, in a given connected component, it is not clear at what number of squares one might expect to find the first example of a 1,1-square-tiled surface. We answer this question by constructing examples of such surfaces in every connected component of every stratum of Abelian differentials. Indeed, our main result is the following.

\begin{theorem}\label{STS}
With the exception of the connected components $\hyp(2g-2)$ and $\hyp(g-1,g-1)$, every connected component of every stratum of Abelian differentials has a 1,1-square-tiled surface with the minimal number of squares required for a square-tiled surface in the ambient stratum. The connected components $\hyp(2g-2)$ and $\hyp(g-1,g-1)$ have 1,1-square-tiled surfaces with $4g-4$ and $4g-2$ squares, respectively. Moreover, these are respectively the minimum number of squares required to construct 1,1-square-tiled surfaces in $\hyp(2g-2)$ and $\hyp(g-1,g-1)$.
\end{theorem}

The theorem demonstrates that, for all but the hyperelliptic components, 1,1-square-tiled surfaces are exhibited in the minimum number of squares possible. The code realising this construction has been included into the \texttt{surface{\_}dynamics} package \cite{D1} of SageMath \cite{sage}.

The extension of this result to every connected component of every stratum of quadratic differentials is work in progress. A direction for further research would be to investigate what this number is for more general $\GL(2,\R)$-invariant suborbifolds.

\subsection*{Ratio-optimising pseudo-Anosovs} Consider the \teichmuller space of marked hyperbolic metrics on the surface $S$, $\T(S)$, equipped with the \teichmuller metric $d_{\T}$, and the curve graph of the surface $S$, $\C(S)$, equipped with the path metric $d_{\C}$. The systole map, $\sys\!:\,\T(S)\to\C(S)$, is a coarsely-defined map that sends a marked hyperbolic metric to the isotopy class of the essential simple closed curve of shortest hyperbolic length. Masur-Minsky {\cite[Consequence of Lemma 2.4]{MM}} showed that there exists a constant $K>0$, depending only on $g$, and a $C\geq 0$ such that $d_{\C}(\sys(X),\sys(Y))\leq K\cdot d_{\T}(X,Y)+C$, for all $X,Y\in{\T(S)}$. In other words, the systole map is {\em coarsely $K$-Lipschitz}. This result was a key step in their proof that $\C(S)$ is $\delta$-hyperbolic. 

It is natural to ask what is the optimum Lipschitz constant, $\kappa_{g}$, defined by
\[\kappa_{g}:=\inf\{K>0\,|\,\exists\,C\geq 0\,\text{such that sys is coarsely }K\text{-Lipschitz}\}.\]
Gadre-Hironaka-Kent-Leininger determined that the ratio of $\kappa_{g}$ to $1/\log(g)$ is bounded from above and below by two positive constants {\cite[Theorem 1.1]{GHKL}}. In such a case, we use the notation $\kappa_{g}\asymp 1/\log(g)$, and say that $\kappa_{g}$ is {\em comparable} to $1/\log(g)$. To find an upper bound for $\kappa_{g}$, Gadre-Hironaka-Kent-Leininger gave a careful version of the proof of Masur-Minsky that $\sys$ is coarsely Lipschitz. They then constructed pseudo-Anosov homeomorphisms for which the ratio $\ell_{\C}(f)/\ell_{\T}(f)\asymp 1/\log(g)$, where $\ell_{\C}(f)$ and $\ell_{\T}(f)$ are the asymptotic translation lengths of $f$ in $\C(S)$ and $\T(S)$, respectively. They then obtained a lower bound for $\kappa_{g}$ by noting that, for any pseudo-Anosov homeomorphism $f$, we have
\[\kappa_{g}\geq\frac{\ell_{\C}(f)}{\ell_{\T}(f)}.\]

Recall that a pair of essential simple closed curves which are in minimal position on a surface $S$ are said to be a {\em filling pair} if the complement of their union is a disjoint collection of disks. Using a Thurston construction on filling pairs, Aougab-Taylor constructed an infinite family of pseudo-Anosov homeomorphisms for which $\tau(f) := \ell_{\T}(f)/\ell_{\C}(f)$ was bounded above by a function $F(g)\asymp\log(g)$ {\cite[Theorem 1.1]{AT2}}; such homeomorphisms are said to be {\em ratio-optimising}. More specifically, given a filling pair $(\alpha,\beta)$ on the surface $S$ with geometric intersection number $i(\alpha,\beta)\asymp g$, they used a Thurston construction on $(\alpha,\beta)$ to construct pseudo-Anosov homeomorphisms for which $\tau(f)\leq\log(D\cdot i(\alpha,\beta))$, where $D$ is a constant independent of $g$. Furthermore, they showed that infinitely many conjugacy classes of primitive ratio-optimising pseudo-Anosov homeomorphisms, produced as above, have their invariant axis contained in the \teichmuller disk $\D(\alpha,\beta)$ of the flat structure determined by the filling pair $(\alpha,\beta)$. The \teichmuller disk $\D(\alpha,\beta)\subset{\T(S)}$ is the image of an embedding of $\Hbb$ determined by the flat structure given by the filling pair $(\alpha,\beta)$.

We observe that the core curves of the cylinders of a 1,1-square-tiled surface form a filling pair with geometric intersection number equal to the number of squares. Hence, as a consequence of Theorem~\ref{STS}, we have the following result.

\begin{theorem}\label{RO}
Given any connected component of any stratum of Abelian differentials, there exist infinitely many conjugacy classes of primitive ratio-optimising pseudo-Anosov homeomorphisms whose invariant axis is contained in the \teichmuller disk of an Abelian differential in that connected component.
\end{theorem}

That is, ratio-optimising pseudo-Anosov homeomorphisms are, in a reasonable sense, ubiquitous in the connected components of strata of Abelian differentials.

\subsection*{Filling pairs on punctured surfaces} Let $S_{g,p}$ denote the surface of genus $g\geq{0}$ with $p\geq{0}$ punctures. We define $i_{g,p}$ to be the minimal geometric intersection number for a filling pair on $S_{g,p}$. The values of $i_{g,p}$ were determined in the works of Aougab-Huang \cite{AH}, Aougab-Taylor \cite{AT1}, and the author \cite{J}. 

For $g\geq 1$, one can ask whether $i_{g,p}$ can be realised as the algebraic intersection number, $\widehat{i}(\alpha,\beta)$, of a filling pair $(\alpha,\beta)$. Aougab-Menasco-Nieland \cite{AMN} answered this question for the case of $i_{g,0}$; that is, for minimally intersecting filling pairs on closed surfaces. Moreover, they were interested in counting the number of mapping class group orbits of such filling pairs. Their method involves algebraically constructing 1,1-square-tiled surfaces with the minimum number of squares in the stratum $\calH(2g-2)$, which they call square-tiled surfaces with connected leaves. The core curves of the cylinders of such surfaces give rise to filling pairs with algebraic intersection number equal to $i_{g,0}$.

Let $n\geq{i_{g,p}}$. By a {\em compatible decomposition} of the surface $S_{g,p}$ into $n+2-2g$ many $4k$-gons, we mean a decomposition of the surface into $4k$-gons $P_{1},\ldots,P_{n+2-2g}$ such that, if $P_{i}$ is a $4(k_{i}+1)$-gon for $k_{i}\geq 0$, then $\sum k_{i} = 2g-2$.

The filling pairs obtained from the cylinders of 1,1-square-tiled surfaces in any stratum of Abelian differentials have algebraic intersection number equal to geometric intersection number and also give rise to a decomposition of the surface into a number of $4k$-gons, with the number of polygons and the number of sides of each polygon depending on the stratum of the square-tiled surface. The resulting set of $4k$-gons form a compatible decomposition of the surface. Indeed, each $4(k_{i}+1)$-gon corresponds to a zero of the Abelian differential of order $k_{i}$.

Using a simple modification of the constructions used in the proof of Theorem~\ref{STS}, we obtain the following result.

\begin{theorem}\label{FP}
Let $n\geq{i_{g,p}}$ and choose a compatible decomposition of $S_{g,p}$ into $n+2-2g$ many $4k$-gons, then there exists a filling pair $(\alpha,\beta)$ on the surface $S_{g,p}$ with
\[\widehat{i}(\alpha,\beta)=i(\alpha,\beta)=n,\]
that gives rise to the specified polygonal decomposition of $S_{g,p}$.
\end{theorem}

This generalises the existence part of the work of Aougab-Menasco-Nieland to the case of any intersection number on any surface $S_{g,p}$.

\subsection*{Sketch of proof of Theorem~\ref{STS}} One might expect that a construction of 1,1-square-tiled surfaces could be achieved by starting with a preferred choice of permutation representative and applying a sequence of Rauzy moves to obtain the desired combinatorics. However, this method is not adequate because the complexity of Rauzy diagrams grow in such a way as to make this extremely computationally difficult. Moreover, the hope that one would be able to easily find such a sequence of Rauzy moves for each connected component is naive. Indeed, the complexity of such a method is demonstrated, for example, in the case of the strata $\calH(2g-5,1,1,1)$ where, with the permutation representatives given by Zorich \cite{Z}, a different sequence of Rauzy moves is required depending on the residue of $2g-5$ modulo 4; see the differing permutation representatives in Proposition~\ref{2g-5} which were obtained by such a method. As such, it seems unreasonable to expect to find a general proof of this nature.

Only in the extremely rigid case of the hyperelliptic components is a proof similar to this achieved. In fact, 1,1-square-tiled surfaces in these components are constructed by hand by adding regular points to the combinatorics given by Rauzy. A method of Margalit relating to minimally intersecting filling pairs on the surface of genus two, referenced in a paper of Aougab-Huang {\cite[Remark 2.18]{AH}}, is then formalised and generalised in order to show that the number of squares achieved for these components is actually the minimum required.

An inductive construction is then adopted in order to build 1,1-square-tiled surfaces in nonhyperelliptic connected components. More specifically, we show that 1,1-square-tiled surfaces in a general connected component can be constructed from 1,1-square-tiled surfaces of lower complexity in such a way that the resulting number of squares and the parity of any resulting spin structure can be easily controlled. We then construct the families of lower complexity 1,1-square-tiled surfaces required to allow this procedure to be completed. A small number of low complexity exceptional cases were found computationally using the \texttt{surface{\_}dynamics} package \cite{D1} of SageMath \cite{sage}.

\subsection*{Plan of the paper}

The combination lemmas that describe how to combine 1,1-square-tiled surfaces, key to the proof of Theorem~\ref{STS}, are given in Subsection~\ref{CL}. Requiring a separate proof method, 1,1-square-tiled surfaces in the hyperelliptic components are constructed first in Subsection~\ref{hyp}. The inductive construction for 1,1-square-tiled surfaces in the remaining connected components is then performed in the rest of Section~\ref{con}. Finally, the proofs of Theorem~\ref{RO} and Theorem~\ref{FP} are given in Section~\ref{RO1} and Section~\ref{FP1}, respectively.

\subsection*{Acknowledgements} We thank Vaibhav Gadre and Tara Brendle for useful discussions and for reading early drafts of this manuscript. We are grateful to the anonymous referee whose comments and suggestions helped to improve this manuscript. We give thanks also to Vincent Delecroix for his help with the \texttt{surface{\_}dynamics} package \cite{D1} of SageMath \cite{sage}. This research was funded by an EPSRC Studentship (EPSRC DTG EP/N509668/1 M{\&}S) while the author was a PhD student at the University of Glasgow.


\section{Moduli space of Abelian differentials}

In this section, we will give the necessary background on Abelian differentials. For more details on this material, we refer the reader to the surveys of Forni-Matheus \cite{FoM} and Zorich \cite{Z1}.

Recall that for $g\geq 2$ we defined the space $\calH$ to be the {\em moduli space of Abelian differentials} on the surface of genus $g$. That is, $\calH$ consists of equivalence classes of pairs $(S,\omega)$ where $S$ is a closed connected Riemann surface of genus $g$ and $\omega$ is a non-zero holomorphic 1-form on $S$, also called an Abelian differential. Two such pairs $(S,\omega)$ and $(S',\omega')$ are equivalent if there exists a biholomorphism $f:S\to S'$ with $f^{*}\omega'=\omega$. When it is appropriate to do so we will simply denote the pair $(S,\omega)$ by either $S$ or $\omega$.

Given an Abelian differential $\omega$ on a Riemann surface $S$, contour integration gives a collection of charts to $\Cbb$ with transition maps given by translations $z\mapsto z+c$. We can then obtain a flat metric on $S$ with cone-type singularities at the zeros of $\omega$ by pulling back the standard metric on $\Cbb$. As such, we can realise the surface $S$ as a finite collection of polygons in $\Cbb$ with pairs of parallel sides of equal length identified by translation, and locally, away from the zeros, the pushforward of $\omega$ gives the standard 1-form ${\rm d}z$ on $\Cbb$. A singularity corresponding to a zero of order $k$ will have a cone-angle of $(k+1)2\pi$ in this metric. The converse also holds. Indeed, given a finite collection of polygons in $\Cbb$ with parallel sides identified by translation, one can define a Riemann surface structure on the surface $S$ obtained from the quotient of these polygons by the side identifications. The local pullback of ${\rm d}z$ will give rise to a well-defined Abelian differential on $S$. Due to this correspondence, points in $\calH$ may also be called {\em translation surfaces}. More specifically, a point is an equivalence class of surfaces equipped with translation structures.

By the Riemann-Roch theorem, the sum of the orders of the zeros of an Abelian differential on a Riemann surface of genus $g$ is equal to $2g-2$ and this data can be used to stratify $\calH$. The stratum $\calH(k_{1},\ldots,k_{n})\subset\calH$, with $k_{i}\geq 1$ and $\sum k_{i} = 2g-2$, is the subset of $\calH$ consisting of Abelian differentials with $n$ distinct zeros of orders $k_{1},\ldots,k_{n}$. Each stratum is an orbifold of complex dimension $2g+n-1$.

The individual strata of $\calH$ may have a number of connected components and the work of Kontsevich-Zorich completely classified these components \cite{KZ}. We will first describe the invariants that they used in their proof before describing the classification itself.

\subsection{Hyperellipticity} 

We say that a translation surface $(S,\omega)$ is {\em hyperelliptic} if there exists an isometric involution $\tau:S\to S$, known as a hyperelliptic involution, that induces a ramified double cover $\pi:S\to S_{0,2g+2}$ from $S$ to the $(2g+2)$-times punctured sphere. Note that we must have $\tau^{*}\omega=-\omega$. Kontsevich-Zorich showed that the strata $\calH(2g-2)$ and $\calH(g-1,g-1)$ contain connected components, denoted by $\hyp(2g-2)$ and $\hyp(g-1,g-1)$ respectively, consisting entirely of hyperelliptic translation surfaces. These connected components will be called the {\em hyperelliptic components}.

We note that the zero of an Abelian differential in $\hyp(2g-2)$ is fixed by the hyperelliptic involution and the two zeros of an Abelian differential in $\hyp(g-1,g-1)$ are symmetric under the hyperelliptic involution.

\subsection{Spin structures and parity} 

The second invariant used to classify the connected components of a stratum is the notion of the parity of a spin structure.

A {\em spin structure} on a Riemann surface $S$ is a choice of half of the canonical class. That is, a choice of divisor class $D\in Pic(S)$ such that
\[2D = K_{S},\]
where $K_{S}$ is the canonical class of $S$. The {\em parity of the spin structure} $D$ is defined to be
\[\dim\Gamma(S,L)\text{ mod } 2,\]
for a line bundle $L$ corresponding to the divisor class $D$, where $\Gamma(S,L)$ is the space of holomorphic sections of the line bundle $L$ on $S$.

Given an Abelian differential $\omega\in\calH(2k_{1},\ldots,2k_{n})$ the divisor
\[Z_{\omega} = 2k_{1}P_{1}+\cdots+2k_{n}P_{n}\]
represents the canonical class $K_{S}$. As such, we have a canonical choice of spin structure on $S$ given by the divisor class
\[D_{\omega} = [k_{1}P_{1}+\cdots+k_{n}P_{n}].\]
Atiyah \cite{A} and Mumford \cite{M} demonstrated that the parity of a spin structure is invariant under continuous deformation. As such, the parity of the canonical spin structure given by an Abelian differential is constant on each connected component of the stratum. We will say that a connected component has {\em even} or {\em odd spin structure} depending on whether or not the parity of $\D_{\omega}$ is 0 or 1.

Recall, that an Abelian differential $\omega$ on $S$ determines a flat metric on $S$ with cone-type singularities. Moreover, this metric has trivial holonomy, and away from the zeros of $\omega$ there is a well-defined horizontal direction. We can therefore define the index, $\ind(\gamma)$, of a simple closed curve $\gamma$ on $S$, avoiding the singularities, to be the degree of the Gauss map of $\gamma$. That is, $\ind(\gamma)$ is the integer such that the total change of angle between the vector tangent to $\gamma$ and the vector tangent to the horizontal direction determined by $\omega$ is $2\pi\cdot \ind(\gamma)$.

Given $\omega\in\calH(2k_{1},\ldots,2k_{n})$, we define a function $\Omega_{\omega}:H_{1}(S,\Z_{2})\to\Z_{2}$ by
\[\Omega_{\omega}([\gamma]) = \ind(\gamma)+1\text{ mod } 2,\]
where $\gamma$ is a simple closed curve and extend to a general homology class by linearity. We claim that this function is well-defined. Firstly, if we homotope a simple closed curve $\gamma$ across a zero of order $k$ then $\ind(\gamma)$ will change by $\pm k$ but since all of our zeros have even order this will fix $\ind(\gamma)$ modulo 2. One can check that for the boundary $\delta$ of a small disk not containing a zero, we have that $\ind(\delta) + 1 \equiv 1 + 1 \equiv 0\mod 2$. For the boundary $\delta$ of a small disk containing a zero of order $k$ we have $\ind(\delta) \equiv (k+1)+1 \equiv 0\mod 2 $ since all of our zeros are of even order. Moreover, it follows from the Poincar{\'e}-Hopf Theorem that $\ind(\delta) + 1 \equiv 0\mod 2$ for any null-homologous simple closed curve $\delta$. Therefore, $\Omega_{\omega}(\mathbf{0}) \equiv 0\mod 2$, and so $\Omega_{\omega}$ is indeed well-defined.

The function $\Omega_{\omega}$ can be shown to be a quadratic form on $H_{1}(S,\Z_{2})$, by which we mean
\[\Omega_{\omega}(a+b) = \Omega_{\omega}(a)+\Omega_{\omega}(b)+a\cdot b,\]
where $a\cdot b$ denotes the standard symplectic intersection form on $H_{1}(S,\Z_{2})$. Now given a choice of representatives $\{\alpha_{i},\beta_{i}\}_{i = 1}^{g}$ of a symplectic basis for $H_{1}(S,\Z_{2})$, we define the {\em Arf invariant} of $\Omega_{\omega}$ to be
\[\sum_{i = 1}^{g}\Omega_{\omega}([\alpha_{i}])\cdot\Omega_{\omega}([\beta_{i}])\text{ mod }2 = \sum_{i = 1}^{g}(\ind(\alpha_{i})+1)(\ind(\beta_{i})+1)\text{ mod }2.\]
Arf \cite{Arf} proved that this number is independent of the choice of symplectic basis and Johnson \cite{Jo} showed that quadratic forms on $H_{1}(S,\Z_{2})$ are in one-to-one correspondence with spin structures on $S$. Moreover, Johnson proved that the value of the Arf invariant of $\Omega_{\omega}$ coincides with the parity of the canonical spin structure determined by $\omega$. We will make use of this formula when we calculate the parity of spin structures later in the paper.

\subsection{Classification of connected components} 

We are now ready to state the classification result of Kontsevich-Zorich. The classification has a stability range and as such the result is given in two parts. 

\begin{theorem}[\cite{KZ}, Theorem 1]
All connected components of strata of Abelian differentials on Riemann surfaces of genus $g\geq 4$ are described by the following list:

The stratum $\calH(2g-2)$ has three connected components: the hyperelliptic one, $\hyp(2g-2)$, and two other components: $\even(2g-2)$ and $\odd(2g-2)$ corresponding to even and odd spin structures.

The stratum $\calH(2l,2l)$, $l\geq 2$, has three connected components: the hyperelliptic one, $\hyp(2l,2l)$, and two other components: $\even(2l,2l)$ and $\odd(2l,2l)$ corresponding to even and odd spin structures.

All other strata of the form $\calH(2l_{1},\ldots,2l_{n})$, $l_{i}\geq 1$, have two connected components: $\even(2l_{1},\ldots,2l_{n})$ and $\odd(2l_{1},\ldots,2l_{n})$ corresponding to even and odd spin structures.

The strata $\calH(2l-1,2l-1)$, $l\geq 2$, has two components: one of them $\hyp(2l-1,2l-1)$ is hyperelliptic; the other $\non(2l-1,2l-1)$ is not.

All other strata of Abelian differentials on Riemann surfaces of genus $g\geq 4$ are nonempty and connected.
\end{theorem}

For lower genera, we have the following classification.

\begin{theorem}[\cite{KZ}, Theorem 2]
The moduli space of Abelian differentials on a Riemann surface of genus $g=2$ contains two strata: $\calH(1,1)$ and $\calH(2)$. Each of them is connected and coincides with its hyperelliptic component.

Each of the strata $\calH(2,2)$ and $\calH(4)$ of the moduli space of Abelian differentials on a Riemann surface of genus $g=3$ has two connected components: the hyperelliptic one, and one having odd spin structure. The other strata are connected for genus $g=3$.
\end{theorem}

\section{Square-tiled surfaces, permutation representatives and filling pair diagrams}

The main objects of study in this paper are square-tiled surfaces. In this section, we introduce these objects along with some associated structures that we will use throughout the paper. Moreover, we will prove a pair of lemmas that will be essential to the construction of 1,1-square-tiled surfaces in Section~\ref{con}.

\subsection{Square-tiled surfaces}

A translation surface is said to be a {\em square-tiled surface} if it is a branched cover of an Abelian differential on the square torus, branched over one point. The polygonal viewpoint for translation surfaces makes this terminology an obvious choice. Indeed, such a surface will be given by a collection of squares in the plane such that the top side (resp. left side) of each square is glued to the bottom side (resp. right side) of another square. See for example the translation surfaces in Figure~\ref{sts}. The minimum number of squares required for a square-tiled surface in the stratum $\calH(k_{1},\ldots,k_{n})$, with $\sum k_{i} = 2g-2$, is $2g+n-2$. Indeed, if we have $s$ squares, then each square contributes 4 sides and so, after identifying sides in pairs, we have a total of $2s$ edges. We must also have at least $n$ vertices after identification and so we see that
\[2-2g = V - E + F \geq n - 2s + s \Rightarrow s \geq 2g+n-2.\]

\begin{figure}[H]
\begin{center}
\begin{tikzpicture}[scale = 1.25]
\draw [line width = 0.3mm, line cap = round] (0,0)--node[left]{0}(0,1);
\draw [line width = 0.3mm, line cap = round] (0,1)--node[above]{1}(1,1);
\draw [line width = 0.3mm, line cap = round] (1,1)--node[above]{2}(2,1);
\draw [line width = 0.3mm, line cap = round] (2,1)--node[above]{3}(3,1);
\draw [line width = 0.3mm, line cap = round] (3,1)--node[right]{0}(3,0);
\draw [line width = 0.3mm, line cap = round] (0,0)--node[below]{3}(1,0);
\draw [line width = 0.3mm, line cap = round] (1,0)--node[below]{2}(2,0);
\draw [line width = 0.3mm, line cap = round] (2,0)--node[below]{1}(3,0);
\draw [densely dashed] (1,0)--(1,1);
\draw [densely dashed] (2,0)--(2,1);
\draw [color = red, line width = 0.3mm] (0.5,0)--(0.5,1);
\draw [color = red, line width = 0.3mm] (2.5,0)--(2.5,1);
\draw [color = green, line width = 0.3mm] (1.5,0)--(1.5,1);
\draw [color = blue, line width = 0.3mm] (0,0.5)--(3,0.5);

\draw [line width = 0.3mm, line cap = round] (5,0)--node[left]{0}(5,1);
\draw [line width = 0.3mm, line cap = round] (5,1)--node[above]{1}(6,1);
\draw [line width = 0.3mm, line cap = round] (6,1)--node[above]{1'}(7,1);
\draw [line width = 0.3mm, line cap = round] (7,1)--node[above]{2}(8,1);
\draw [line width = 0.3mm, line cap = round] (8,1)--node[above]{3}(9,1);
\draw [line width = 0.3mm, line cap = round] (9,1)--node[right]{0}(9,0);
\draw [line width = 0.3mm, line cap = round] (9,0)--node[below]{1'}(8,0);
\draw [line width = 0.3mm, line cap = round] (8,0)--node[below]{1}(7,0);
\draw [line width = 0.3mm, line cap = round] (7,0)--node[below]{2}(6,0);
\draw [line width = 0.3mm, line cap = round] (6,0)--node[below]{3}(5,0);
\draw [densely dashed] (6,0)--(6,1);
\draw [densely dashed] (7,0)--(7,1);
\draw [densely dashed] (8,0)--(8,1);
\draw [color = red, line width = 0.3mm] (5.5,0)--(5.5,1);
\draw [color = red, line width = 0.3mm] (6.5,0)--(6.5,1);
\draw [color = red, line width = 0.3mm] (7.5,0)--(7.5,1);
\draw [color = red, line width = 0.3mm] (8.5,0)--(8.5,1);
\draw [color = blue, line width = 0.3mm] (5,0.5)--(9,0.5);
\end{tikzpicture}
\end{center}
\caption[Two square-tiled surfaces in $\calH(2)$.]{Two square-tiled surfaces in $\calH(2)$ with a single horizontal cylinder. The surface on the right also has a single vertical cylinder while the one on the left has two vertical cylinders.}\label{sts}
\end{figure}
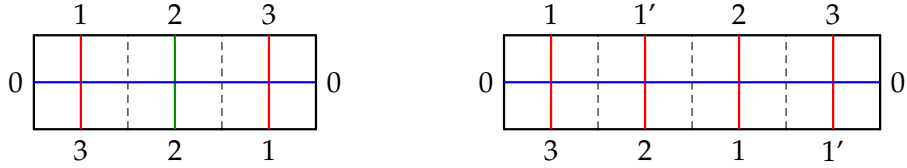

One important piece of combinatorial data for a square-tiled surface is the number of maximal, flat horizontal or vertical cylinders. A {\em cylinder} is a maximal embedded annulus in the surface, not containing any singularities in its interior. For example, the curves between the sides labelled by 0s in Figure~\ref{sts} are the core curves of the horizontal cylinders of the surfaces. One can also see that the surface on the left also has two vertical cylinders while the one on the right has a single vertical cylinder. If a square-tiled surface has a single vertical cylinder and a single horizontal cylinder then we shall call it a {\em 1,1-square-tiled surface}. Since we are interested in constructing 1,1-square-tiled surfaces using the minimum possible number of squares, the vertical and horizontal cylinders will both have height one.

The process of splitting a pair of identified sides into two and identifying them as before adds a marked point, a zero of order zero, to the translation surface. Observe that the surface on the right of Figure~\ref{sts} is obtained from the surface on the left by performing such an operation on the sides labelled 1. That is, we split the side labelled 1 into two side labelled by 1 and 1'. This does not change the connected component of the surface and we will make use of this technique when adding squares to hyperelliptic square-tiled surfaces in Subsection~\ref{hyp}.

We briefly recall an algebraic way of describing square-tiled surfaces. Firstly, we number each square in the surface from 1 to $n$. We then define two elements $h$ and $v$ of the symmetric group $\Sigma_{n}$ as follows. The image of $i$ under the element $h$ is the number of the square that is glued to the right of the square numbered $i$. The image of $i$ under the element $v$ is the number of the square that is glued to the top of the square numbered $i$. The square-tiled surface is then represented by the pair $(h,v)$ up to simultaneous conjugation of $h$ and $v$, and the information about the stratum of the surface is contained in the commutator $[h,v]=hvh^{-1}v^{-1}$. That is, if $[h,v]$ is a product of disjoint $(k_{i}+1)$-cycles, $1\leq i\leq n$, then the square-tiled surface lies in the stratum $\calH(k_{1},\ldots,k_{n})$. For example, the surface in $\calH(2)$ on the left of Figure~\ref{sts} is represented by the pair $(h,v)$, where $h = (1,2,3)$ and $v = (1,3)(2)$, and we see that $[h,v] = (1,3,2)$ is a 3-cycle. We see here again that the minimum number of squares required for a square-tiled surface in $\calH(k_{1},\ldots,k_{n})$ is $2g+n-2$. Indeed, the number of squares in such a surface is at least the size of the support of the commutator which is $\Sigma(k_{i}+1) = 2g+n-2$.

It is this connection to the symmetric group that was utilised by Eskin-Okounkov in their asymptotic counts. We remark that while the counting of square-tiled surfaces having only a single horizontal cylinder is well suited to the representation theory of the symmetric group the counting of 1,1-square-tiled surfaces is not. Indeed, in the former case, and by the above algebraic description, a square-tiled surface having one horizontal cylinder and lying in a particular stratum corresponds to finding two $n$-cycles whose product lies in a specified conjugacy class. In the latter case, however, we are required to find two $n$-cycles whose commutator lies in a particular conjugacy class.

We discussed the algebraic representation of square-tiled surfaces here for completeness. In the sequel, our results will be stated and proved using the permutation representatives introduced in the next subsection. There we also describe for square-tiled surfaces with a single horizontal cylinder how to relate the two notions.

\subsection{Permutation representatives}\label{permreps}

An {\em interval exchange transformation} is a self map of the interval that divides the interval into subintervals and then permutes them. Consider the translation surface given in Figure~\ref{IET}. The first return map to the horizontal transversal $T$ under the upwards vertical flow on the surface induces an interval exchange transformation on $T$ whose permutation is
\begin{equation}\label{iet}
\Pi = \left(\begin{matrix}
0&1&2&3&4 \\
4&3&1&2&0
\end{matrix}\right).
\end{equation}
That is, under the upwards vertical flow, the interval on $T$ lying below the side labelled 0 returns in position 4 (counting from the left and starting at 0), and so on. In general, the interval below side $i$ returns in position $\Pi^{-1}(i)$. For more details on the connections between translation surfaces and interval exchange transformations we direct the reader to the survey of Yoccoz~\cite{Y}.

The {\em extended Rauzy class} of this permutation is a class of permutations related under an induction method for interval exchange transformations introduced by Rauzy \cite{R}. Another choice of transversal will give an interval exchange transformation whose permutation lies in the extended Rauzy class of permutation (\ref{iet}). Conversely, any translation surface obtained as a suspension of an interval exchange transformation whose permutation lies in the same extended Rauzy class as permutation (\ref{iet}) will lie in the same connected component of a stratum as the translation surface in Figure~\ref{IET}. Indeed, Veech showed that extended Rauzy classes are in one-to-one correspondence with the connected components of strata \cite{V1}. As such, any choice of permutation in an extended Rauzy class will be called a {\em permutation representative} of the stratum component containing the associated translation surface.

\begin{figure}[H]
\begin{center}
\begin{tikzpicture}[scale = 0.8]
\draw [line width = 0.3mm, line cap = round] (0,0)--(1,3);
\draw (0.5-0.3,1.5+0.1) node{0};
\draw [line width = 0.3mm, line cap = round] (1,3)--(2,3.2);
\draw (1.5,3.5) node{1};
\draw [line width = 0.3mm, line cap = round] (2,3.2)--(3.4,2);
\draw (2.9,2.9) node{2};
\draw [line width = 0.3mm, line cap = round] (3.4,2)--(4.5,3);
\draw (3.95-0.25,2.5+0.275) node{3};
\draw [line width = 0.3mm, line cap = round] (4.5,3)--(6,1);
\draw (0.75+0.3+4.5,-1+0.225+3) node{4};
\draw [line width = 0.3mm, line cap = round] (6,1)--(5,-2);
\draw (0.5+0.3+5,1.5-0.1-2) node{0};
\draw [line width = 0.3mm, line cap = round] (2.6,-1)--(3.6,-0.8);
\draw (3.1,-1.3) node{1};
\draw [line width = 0.3mm, line cap = round] (5,-2)--(3.6,-0.8);
\draw (4.1,-1.7) node{2};
\draw [line width = 0.3mm, line cap = round] (2.6,-1)--(1.5,-2);
\draw (3.95+0.25-1.9,2.5-0.275-4) node{3};
\draw [line width = 0.3mm, line cap = round] (1.5,-2)--(0,0);
\draw (0.75-0.3,-1-0.225) node{4};
\draw (0,0)--(6,0);
\draw (6,0) node[right]{$T$};
\draw [densely dashed] (1,0)--(1,3);
\draw [densely dashed] (2,0)--(2,3.2);
\draw [densely dashed] (3.4,0)--(3.4,2);
\draw [densely dashed] (4.5,0)--(4.5,3);
\draw [densely dashed] (5,0)--(5,-2);
\draw [densely dashed] (3.6,0)--(3.6,-0.8);
\draw [densely dashed] (2.6,0)--(2.6,-1);
\draw [densely dashed] (1.5,0)--(1.5,-2);
\end{tikzpicture}
\end{center}
\caption[The first return map to the horizontal transversal $T$.]{The first return map to the horizontal transversal $T$ under the vertical flow induces an interval exchange transformation.}\label{IET}
\end{figure}
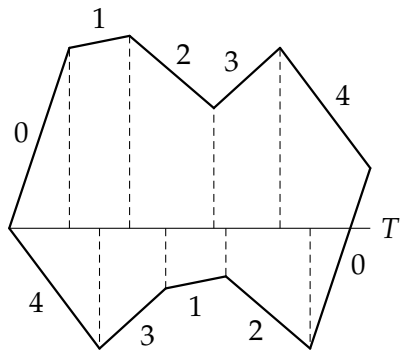

If the first symbol of the top row of a permutation representative is equal to the last symbol of the bottom row, then it is possible to construct a translation surface in the associated stratum component having a single horizontal cylinder. Zorich constructed permutation representatives of this form for every connected component of every stratum of Abelian and quadratic differentials \cite{Z}.

We will be interested in the construction of a square-tiled surface from such a permutation representative. To do this, one takes a line of squares of length one less than the number of symbols in the permutation representative and labels the left and right sides of this line of squares with the first symbol in the top row and last symbol of the bottom row, respectively. From our assumption, these symbols are the same and so the resulting square-tiled surface will have a single horizontal cylinder. We then label the top sides (resp. bottom sides) with the remaining symbols from the top row (resp. bottom row). For example, the square-tiled surface on the right of Figure~\ref{sts} is the one obtained, up to a relabelling, by performing this construction using permutation ({\ref{iet}) and we see that it does have a single horizontal cylinder, as claimed. The surface on the left of Figure~\ref{sts} can be obtained from the permutation
\begin{equation}\label{iet2}
\begin{pmatrix}
0&1&2&3 \\
3&2&1&0
\end{pmatrix}.
\end{equation}
From here on, we will call a square-tiled surface constructed in this manner the square-tiled surface represented by the associated permutation representative.

As briefly mentioned above, we remark that attempting to use Rauzy moves to search the extended Rauzy classes of these permutations for permutations representing 1,1-square-tiled surfaces is not a feasible method for solving our problem. Indeed, Delecroix showed that the cardinality of extended Rauzy classes increases in such a way that this task would be incredibly computationally intensive \cite{D}. Moreover, it is unlikely that examples for different strata could be found using similar sequences of Rauzy moves, and so a general proof of this nature would be difficult to find.

If we have a square-tiled surface with a single horizontal cylinder constructed as above, then, assuming the 0s to be the first symbol of the top row and last symbol of the bottom row, information concerning the vertical cylinders is contained in the permutation obtained by removing the 0s from each row of the permutation representative. Indeed, if the vertical cylinders have height one, then the number of vertical cylinders of the surface is equal to the number of cycles of this permutation. For example, under this modification, permutation (\ref{iet2}) becomes the permutation $(1,3)(2)$ and indeed the surface on the left of Figure~\ref{sts} has two vertical cylinders. Since we are interested in 1,1-square-tiled surfaces with a minimal number of squares, we will want the vertical cylinder to have height one, and so want this modified permutation to be a cyclic permutation. Indeed, for permutation (\ref{iet}) we obtain $(1,4,2,3)$ and it can be checked that the surface on the right of Figure~\ref{sts} does indeed have a single vertical cylinder. In fact, what we have just described is the following procedure: if we remove the 0s from a permutation representative $\pi$ of a square-tiled surface with one horizontal cylinder to obtain the permutation $\overline{\pi}$, then the square-tiled surface is represented algebraically by $(h,v)$ with $h = (1,2,\ldots,n)$ and $v = \overline{\pi}^{-1}$. This process is reversible in the obvious way.

An algorithm that can be used to determine the spin parity of a translation surface associated to a given permutation representative is described by Zorich~\cite[Appendix C]{Z}. This algorithm uses the permutation representative to calculate the intersection matrix of a set of generating cycles for the homology of the surface and for which the index of each cycle is known. One then performs linear algebra to reduce this set to a basis and the Arf invariant is then calculated. By the discussion in the previous paragraph, this intersection matrix and set of generating cycles can also be obtained from the vertical permutation $v$ of a square-tiled surface $(h,v)$ having a single horizontal cylinder. The linear algebra then follows as before. In general, the permutations $h$ and $v$ are not well adapted to handling the homology of the square-tiled surface and so there is not a simple way to derive the spin parity of the surface from these permutations alone. It is also easy to calculate the spin parity of a permutation representative using the \texttt{.arf\_invariant()} method for permutation representatives available inside the \texttt{surface\_dynamics}~\cite{D1} package of SageMath~\cite{sage}.

Note that adding a marked point to a side represented by label $x$, as described in the previous subsection, corresponds to adding a label to the right of $x$ in both rows of the permutation representative.

\subsection{Filling pair diagrams}\label{fpd}

On the surface $S$ of genus $g$, a pair of essential simple closed curves $\alpha$ and $\beta$ which are in minimal position, that is $i(\alpha,\beta):=\min_{\gamma\in[\alpha]}|\gamma\cap\beta|=|\alpha\cap\beta|$, are said to be a {\em filling pair} if their complement is a disjoint union of disks. We note that the core curves of the vertical and horizontal cylinders of a 1,1-square-tiled surface form a filling pair on that surface. Since we have an Abelian differential, all intersections occur with the same orientation. Moreover, each complementary region is a $4k$-gon and corresponds to a zero of order $k-1$ of the associated Abelian differential.

Conversely, given a filling pair on the surface $S$ of genus $g$ whose intersections all occur with the same orientation, the dual complex of the filling pair is a square complex and we can realise the surface as a collection of squares in the plane with sides identified by translations, in other words, as a square-tiled surface. As above, each complementary region with $4k$ sides will give rise to a zero of the Abelian differential of order $k-1$.

With this correspondence in mind, we can form a ribbon graph from the vertical and horizontal cylinders of a 1,1-square-tiled surface. We shall call the underlying oriented combinatorial graph a {\em filling pair diagram} for the associated square-tiled surface.

\begin{figure}[H]
\begin{center}
{\footnotesize
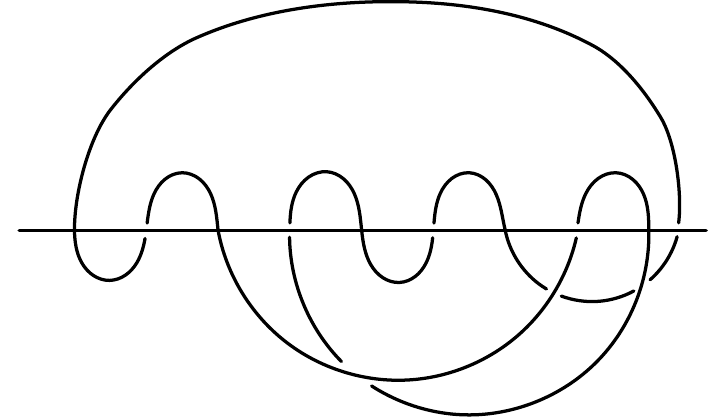
\vspace{0.5cm}
}
\end{center}
\caption{A filling pair diagram corresponding to a square-tiled surface represented by permutation (\ref{fpdperm}).}\label{fpdex}
\end{figure}

We will explain this construction by means of an example. Indeed, consider the 1,1-square-tiled surface in $\calH(4)$ with permutation representative
\begin{equation}\label{fpdperm}
\left(\begin{matrix}
0&1&2&3&4&5 \\
2&5&4&1&3&0
\end{matrix}\right).
\end{equation}
We first draw a horizontal line corresponding to the horizontal cylinder. Note that we will think of the ends being identified even though we do not join them in the diagram. We then add one vertex to the line for every square in the surface, equivalently for every non-zero symbol in the permutation representative. We then add an edge joining the bottom of the vertex corresponding to label $x$ to the top of the vertex corresponding to the label $y$ if and only if $y$ lies below $x$ in the permutation obtained by removing the 0s. The concatenation of these edges represents the vertical cylinder on the surface. The filling pair diagram associated to permutation (\ref{fpdperm}) is shown in Figure~\ref{fpdex}. Note that, given a filling pair diagram drawn as above, the reverse of this process allows us to easily construct the permutation representative. In Figure~\ref{fpdex} we have added the labels from the permutation representative to clarify the construction. In the remainder of this paper, we will not add the labels to the vertices of filling pair diagrams.

An alternative description of a filling pair diagram is afforded by the algebraic description of a square-tiled surface. In this case, the filling pair diagram associated to the 1,1-square-tiled surface $(h,v)$ can be described as the graph with vertex set $\{1,2,\ldots,n\}$ and edges $(i,h(i))$ and $(i,v(i))$ for all $1\leq i\leq n$, where we choose to draw the edges determined by $h$ as a horizontal line.

Taking a regular neighbourhood of the filling pair diagram gives a ribbon graph with one boundary component for every complementary region of the filling pair. We can construct the square-tiled surface by gluing a saddle with cone-angle $2k\pi$ onto every boundary component with $4k$ sides. Indeed, the ribbon graph obtained from the filling pair diagram in Figure~\ref{fpdex} has a single boundary component with 20 sides and so we glue in a saddle with cone-angle $10\pi = (4+1)2\pi$. As such, the square-tiled surface will have a single zero of order 4. This agrees with the fact that the permutation representative corresponded to a surface in $\calH(4)$.

\subsection{Combination lemmas}\label{CL}

We now provide the combination lemmas that will be crucial to the construction of 1,1-square-tiled surfaces in the next section. The first lemma describes how to combine two 1,1-square-tiled surfaces to produce a single 1,1-square-tiled surface of higher complexity. The second lemma describes how the parity of the spin structure of a surface built in this way depends on the parities of the spin structures of the constituent surfaces.

Before proving the first lemma we will demonstrate the construction through an example. First, consider the permutation
\[\left(\begin{matrix}
0&1&2&3&4&5 \\
2&5&4&1&3&0
\end{matrix}\right)\]
representing a 1,1-square-tiled surface in $\calH(4)$ with the minimum number of squares. We will describe the process of combining this surface with itself, as in Lemma ~\ref{comb1} below, to produce a 1,1-square-tiled surface in $\calH(4,4)$. The key property we want for this construction is that, on the square-tiled surface represented by this permutation, the bottom of the first square is identified with the top of the second. This can be seen in the permutation if the top row starts 0,1,2, and the second row starts with a 2.

\begin{center}
\begin{figure}[H]
\begin{center}
\begin{tikzpicture}[scale = 1.25]
\draw [line width = 0.3mm, line cap = round] (0,0)--node[left]{0}(0,1);
\draw [line width = 0.3mm, line cap = round] (0,1)--node[above]{1}(1,1);
\draw [line width = 0.3mm, line cap = round] (1,1)--node[above]{2}(2,1);
\draw [line width = 0.3mm, line cap = round] (2,1)--node[above]{3}(3,1);
\draw [line width = 0.3mm, line cap = round] (3,1)--node[above]{4}(4,1);
\draw [line width = 0.3mm, line cap = round] (4,1)--node[above]{5}(5,1);
\draw [line width = 0.3mm, line cap = round] (5,1)--node[right]{0}(5,0);
\draw [line width = 0.3mm, line cap = round] (5,0)--node[below]{3}(4,0);
\draw [line width = 0.3mm, line cap = round] (4,0)--node[below]{1}(3,0);
\draw [line width = 0.3mm, line cap = round] (3,0)--node[below]{4}(2,0);
\draw [line width = 0.3mm, line cap = round] (2,0)--node[below]{5}(1,0);
\draw [line width = 0.3mm, line cap = round] (1,0)--node[below]{2}(0,0);
\draw [densely dashed](1,0)--(1,1);
\draw [densely dashed](2,0)--(2,1);
\draw [densely dashed](3,0)--(3,1);
\draw [densely dashed](4,0)--(4,1);
\draw [line width = 0.3mm, color = red] (0,0.5)--(5,0.5);
\draw [line width = 0.3mm, color = blue] (0.5,0)--(0.5,1);
\draw [line width = 0.3mm, color = blue] (1.5,0)--(1.5,1);
\draw [line width = 0.3mm, color = blue] (2.5,0)--(2.5,1);
\draw [line width = 0.3mm, color = blue] (3.5,0)--(3.5,1);
\draw [line width = 0.3mm, color = blue] (4.5,0)--(4.5,1);
\draw [line width = 0.3mm, color = green] (0,0) -- (1,1);
\foreach \x in {0,1,2,3,4,5}
	{\node[color = blue] at (\x,0) {$\bullet$};}
\foreach \x in {0,1,2,3,4,5}
	{\node[color = blue] at (\x,1) {$\bullet$};}

\draw [line width = 0.3mm, line cap = round] (0+6,0)--node[left]{0'}(0+6,1);
\draw [line width = 0.3mm, line cap = round] (0+6,1)--node[above]{1'}(1+6,1);
\draw [line width = 0.3mm, line cap = round] (1+6,1)--node[above]{2'}(2+6,1);
\draw [line width = 0.3mm, line cap = round] (2+6,1)--node[above]{3'}(3+6,1);
\draw [line width = 0.3mm, line cap = round] (3+6,1)--node[above]{4'}(4+6,1);
\draw [line width = 0.3mm, line cap = round] (4+6,1)--node[above]{5'}(5+6,1);
\draw [line width = 0.3mm, line cap = round] (5+6,1)--node[right]{0'}(5+6,0);
\draw [line width = 0.3mm, line cap = round] (5+6,0)--node[below]{3'}(4+6,0);
\draw [line width = 0.3mm, line cap = round] (4+6,0)--node[below]{1'}(3+6,0);
\draw [line width = 0.3mm, line cap = round] (3+6,0)--node[below]{4'}(2+6,0);
\draw [line width = 0.3mm, line cap = round] (2+6,0)--node[below]{5'}(1+6,0);
\draw [line width = 0.3mm, line cap = round] (1+6,0)--node[below]{2'}(0+6,0);
\draw [densely dashed](1+6,0)--(1+6,1);
\draw [densely dashed](2+6,0)--(2+6,1);
\draw [densely dashed](3+6,0)--(3+6,1);
\draw [densely dashed](4+6,0)--(4+6,1);
\draw [line width = 0.3mm, color = red] (6,0.5)--(11,0.5);
\draw [line width = 0.3mm, color = blue] (6.5,0)--(6.5,1);
\draw [line width = 0.3mm, color = blue] (7.5,0)--(7.5,1);
\draw [line width = 0.3mm, color = blue] (8.5,0)--(8.5,1);
\draw [line width = 0.3mm, color = blue] (9.5,0)--(9.5,1);
\draw [line width = 0.3mm, color = blue] (10.5,0)--(10.5,1);
\draw [line width = 0.3mm, color = green] (6,0) -- (7,1);
\foreach \x in {0,1,2,3,4,5}
	{\node[color = red] at (6+\x,0) {$\blacklozenge$};}
\foreach \x in {0,1,2,3,4,5}
	{\node[color = red] at (6+\x,1) {$\blacklozenge$};}
\end{tikzpicture}
\end{center}
\begin{center}
\begin{tikzpicture}[scale = 1.25]
\draw [line width = 0.3mm, line cap = round] (0,0)--node[left]{0'}(0,1);
\draw [line width = 0.3mm, line cap = round] (0,1)--node[above]{1'}(1,1);
\draw [line width = 0.3mm, line cap = round] (1,1)--node[above]{2}(2,1);
\draw [line width = 0.3mm, line cap = round] (2,1)--node[above]{3}(3,1);
\draw [line width = 0.3mm, line cap = round] (3,1)--node[above]{4}(4,1);
\draw [line width = 0.3mm, line cap = round] (4,1)--node[above]{5}(5,1);
\draw [line width = 0.3mm, line cap = round] (5,0)--node[below]{3}(4,0);
\draw [line width = 0.3mm, line cap = round] (4,0)--node[below]{1}(3,0);
\draw [line width = 0.3mm, line cap = round] (3,0)--node[below]{4}(2,0);
\draw [line width = 0.3mm, line cap = round] (2,0)--node[below]{5}(1,0);
\draw [line width = 0.3mm, line cap = round] (1,0)--node[below]{2}(0,0);
\draw [line width = 0.3mm, line cap = round] (0+5,1)--node[above]{1}(1+5,1);
\draw [line width = 0.3mm, line cap = round] (1+5,1)--node[above]{2'}(2+5,1);
\draw [line width = 0.3mm, line cap = round] (2+5,1)--node[above]{3'}(3+5,1);
\draw [line width = 0.3mm, line cap = round] (3+5,1)--node[above]{4'}(4+5,1);
\draw [line width = 0.3mm, line cap = round] (4+5,1)--node[above]{5'}(5+5,1);
\draw [line width = 0.3mm, line cap = round] (5+5,1)--node[right]{0'}(5+5,0);
\draw [line width = 0.3mm, line cap = round] (5+5,0)--node[below]{3'}(4+5,0);
\draw [line width = 0.3mm, line cap = round] (4+5,0)--node[below]{1'}(3+5,0);
\draw [line width = 0.3mm, line cap = round] (3+5,0)--node[below]{4'}(2+5,0);
\draw [line width = 0.3mm, line cap = round] (2+5,0)--node[below]{5'}(1+5,0);
\draw [line width = 0.3mm, line cap = round] (1+5,0)--node[below]{2'}(0+5,0);
\draw [densely dashed](1,0)--(1,1);
\draw [densely dashed](2,0)--(2,1);
\draw [densely dashed](3,0)--(3,1);
\draw [densely dashed](4,0)--(4,1);
\draw [densely dashed](5,0)--(5,1);
\draw [densely dashed](1+5,0)--(1+5,1);
\draw [densely dashed](2+5,0)--(2+5,1);
\draw [densely dashed](3+5,0)--(3+5,1);
\draw [densely dashed](4+5,0)--(4+5,1);
\draw [line width = 0.3mm, color = red] (0,0.5)--(10,0.5);
\draw [line width = 0.3mm, color = blue] (0.5,0)--(0.5,1);
\draw [line width = 0.3mm, color = blue] (1.5,0)--(1.5,1);
\draw [line width = 0.3mm, color = blue] (2.5,0)--(2.5,1);
\draw [line width = 0.3mm, color = blue] (3.5,0)--(3.5,1);
\draw [line width = 0.3mm, color = blue] (4.5,0)--(4.5,1);
\draw [line width = 0.3mm, color = blue] (5.5,0)--(5.5,1);
\draw [line width = 0.3mm, color = blue] (6.5,0)--(6.5,1);
\draw [line width = 0.3mm, color = blue] (7.5,0)--(7.5,1);
\draw [line width = 0.3mm, color = blue] (8.5,0)--(8.5,1);
\draw [line width = 0.3mm, color = blue] (9.5,0)--(9.5,1);
\draw [line width = 0.3mm, color = green] (0,0)--(1,1);
\draw [line width = 0.3mm, color = green] (5,0)--(6,1);
\foreach \x in {1,2,3,4,5}
	{\node[color = blue] at (\x,0) {$\bullet$};}
\foreach \x in {-6,0,1,2,3,4}
	{\node[color = red] at (6+\x,0) {$\blacklozenge$};}
\foreach \x in {2,3,4,5,6}
	{\node[color = blue] at (\x,1) {$\bullet$};}
\foreach \x in {-6,-5,1,2,3,4}
	{\node[color = red] at (6+\x,1) {$\blacklozenge$};}
\end{tikzpicture}
\end{center}
\caption{Example of the cylinder concatenation of two 1,1-square-tiled surfaces in $\calH(4)$.}\label{combex}
\end{figure}
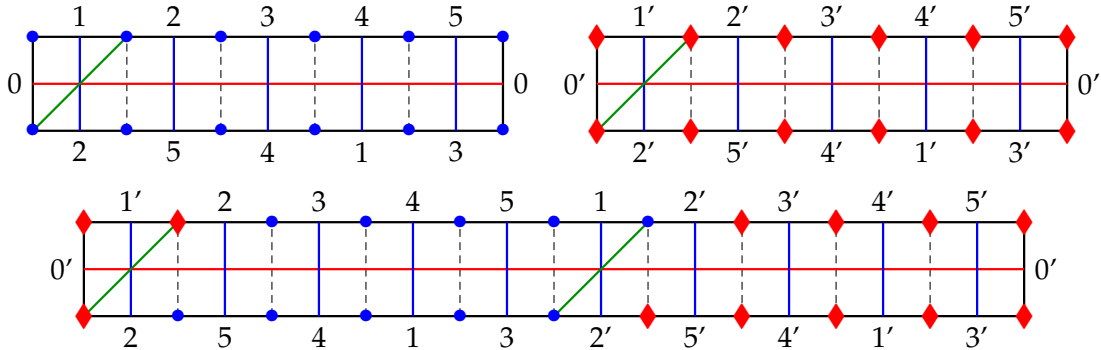
\end{center}

We first realise two copies of this surface as in the top of Figure~\ref{combex}. We then cut both surfaces open along the diagonal lines shown in Figure~\ref{combex} and glue the left-side of this slit on the one surface to the right-side of the slit on the other. This has the effect of concatenating the horizontal and vertical cylinders of the two surfaces. Indeed, we obtain the surface at the bottom of Figure~\ref{combex} which can be seen to be a 1,1-square-tiled surface. It can be checked that the surface lies in $\calH(4,4)$. After relabelling, the permutation representative for this surface is
\[\left(\begin{matrix}
0&1&2&3&4&5&6&7&8&9&10 \\
2&5&4&6&3&7&10&9&1&8&0
\end{matrix}\right).\]
The construction can easily be performed directly on the permutation representatives and it is also easy to see what the process involves for filling pair diagrams. In the following lemmas, we will call this process of combining surfaces {\em cylinder concatenation}.

We observe that the zeros of the constituent surfaces were preserved and that, by using 1,1-square-tiled surfaces with the minimal number of squares required for their respective strata, we obtained a 1,1-square-tiled surface with the minimum number of squares for its stratum. That this is true in general is the content of the following lemma.

\begin{lemma}\label{comb1}
Suppose the permutations
\[\left(\begin{matrix}
0&1&2&\cdots&\cdots& N \\
2& \cdots&\cdots& \cdots &\cdots & 0
\end{matrix}\right),\]
and
\[\left(\begin{matrix}
0'&1'&2'&\cdots&\cdots& M' \\
2'& \cdots&\cdots& \cdots &\cdots & 0'
\end{matrix}\right)\]
represent 1,1-square-tiled surfaces $S_{1}$ and $S_{2}$ in the strata $\calH_{g_{1}}(k_{1},\ldots,k_{n})$ and $\calH_{g_{2}}(l_{1},\ldots,l_{m})$, respectively. We assume only that they have first rows beginning 0, 1, 2 (resp. 0', 1', 2') and second rows beginning with 2 (resp. 2'). Then the 1,1-square-tiled surface $S$ obtained from these surfaces by the cylinder concatenation method lies in $\calH_{g_{1}+g_{2}-1}(k_{1},\ldots,k_{n},l_{1},\ldots,l_{m})$. Moreover, if $S_{1}$ and $S_{2}$ have the minimum number of squares for their respective strata then so does $S$.
\end{lemma}

\begin{proof}
Note that the surfaces $S_{1}$ and $S_{2}$ can be realised as in the top of Figure~\ref{GL1}. Since the sides labelled by 2 (resp. 2') are diagonally opposite, we can construct the diagonal green curves of slope 1 shown in each surface. The process of cylinder concatenation for these surfaces is the process of cutting each surface open along these diagonal curves and gluing the right side of the slit in each surface to the left side of the slit in the other surface. This action concatenates the cylinders as expected and so we do indeed produce a 1,1-square-tiled surface $S$. Moreover, it can be seen that the that the zeros of $S$ are exactly the union of the zeros of $S_{1}$ and $S_{2}$. Indeed, we need only the check that the angles around the vertices at the diagonal lines are preserved which is easily checked. By an Euler characteristic argument, the genus of $S$ is $g_{1}+g_{2}-1$. That is, $S$ lies in $\calH_{g_{1}+g_{2}-1}(k_{1},\ldots,k_{n},l_{1},\ldots,l_{m})$, as claimed.
\end{proof}

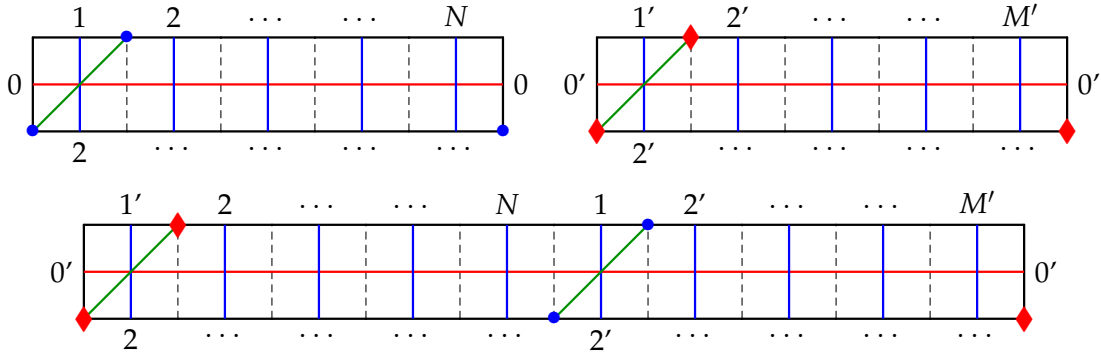
\begin{figure}[H]
\begin{center}
\begin{tikzpicture}[scale = 1.25]
\draw [line width = 0.3mm, line cap = round] (0,0)--node[left]{0}(0,1);
\draw [line width = 0.3mm, line cap = round] (0,1)--node[above]{1}(1,1);
\draw [line width = 0.3mm, line cap = round] (1,1)--node[above]{2}(2,1);
\draw [line width = 0.3mm, line cap = round] (2,1)--node[above]{$\cdots$}(3,1);
\draw [line width = 0.3mm, line cap = round] (3,1)--node[above]{$\cdots$}(4,1);
\draw [line width = 0.3mm, line cap = round] (4,1)--node[above]{$N$}(5,1);
\draw [line width = 0.3mm, line cap = round] (5,1)--node[right]{0}(5,0);
\draw [line width = 0.3mm, line cap = round] (5,0)--node[below]{$\cdots$}(4,0);
\draw [line width = 0.3mm, line cap = round] (4,0)--node[below]{$\cdots$}(3,0);
\draw [line width = 0.3mm, line cap = round] (3,0)--node[below]{$\cdots$}(2,0);
\draw [line width = 0.3mm, line cap = round] (2,0)--node[below]{$\cdots$}(1,0);
\draw [line width = 0.3mm, line cap = round] (1,0)--node[below]{2}(0,0);
\draw [densely dashed](1,0)--(1,1);
\draw [densely dashed](2,0)--(2,1);
\draw [densely dashed](3,0)--(3,1);
\draw [densely dashed](4,0)--(4,1);
\draw [line width = 0.3mm, color = red] (0,0.5)--(5,0.5);
\draw [line width = 0.3mm, color = blue] (0.5,0)--(0.5,1);
\draw [line width = 0.3mm, color = blue] (1.5,0)--(1.5,1);
\draw [line width = 0.3mm, color = blue] (2.5,0)--(2.5,1);
\draw [line width = 0.3mm, color = blue] (3.5,0)--(3.5,1);
\draw [line width = 0.3mm, color = blue] (4.5,0)--(4.5,1);
\draw [line width = 0.3mm, color = green] (0,0)--(1,1);
\foreach \x in {0,5}
	{\node[color = blue] at (\x,0) {$\bullet$};}
\foreach \x in {1}
	{\node[color = blue] at (\x,1) {$\bullet$};}

\draw [line width = 0.3mm, line cap = round] (0+6,0)--node[left]{0'}(0+6,1);
\draw [line width = 0.3mm, line cap = round] (0+6,1)--node[above]{1'}(1+6,1);
\draw [line width = 0.3mm, line cap = round] (1+6,1)--node[above]{2'}(2+6,1);
\draw [line width = 0.3mm, line cap = round] (2+6,1)--node[above]{$\cdots$}(3+6,1);
\draw [line width = 0.3mm, line cap = round] (3+6,1)--node[above]{$\cdots$}(4+6,1);
\draw [line width = 0.3mm, line cap = round] (4+6,1)--node[above]{$M'$}(5+6,1);
\draw [line width = 0.3mm, line cap = round] (5+6,1)--node[right]{0'}(5+6,0);
\draw [line width = 0.3mm, line cap = round] (5+6,0)--node[below]{$\cdots$}(4+6,0);
\draw [line width = 0.3mm, line cap = round] (4+6,0)--node[below]{$\cdots$}(3+6,0);
\draw [line width = 0.3mm, line cap = round] (3+6,0)--node[below]{$\cdots$}(2+6,0);
\draw [line width = 0.3mm, line cap = round] (2+6,0)--node[below]{$\cdots$}(1+6,0);
\draw [line width = 0.3mm, line cap = round] (1+6,0)--node[below]{2'}(0+6,0);
\draw [densely dashed](1+6,0)--(1+6,1);
\draw [densely dashed](2+6,0)--(2+6,1);
\draw [densely dashed](3+6,0)--(3+6,1);
\draw [densely dashed](4+6,0)--(4+6,1);
\draw [line width = 0.3mm, color = red] (6,0.5)--(11,0.5);
\draw [line width = 0.3mm, color = blue] (6.5,0)--(6.5,1);
\draw [line width = 0.3mm, color = blue] (7.5,0)--(7.5,1);
\draw [line width = 0.3mm, color = blue] (8.5,0)--(8.5,1);
\draw [line width = 0.3mm, color = blue] (9.5,0)--(9.5,1);
\draw [line width = 0.3mm, color = blue] (10.5,0)--(10.5,1);
\draw [line width = 0.3mm, color = green] (6,0)--(7,1);
\foreach \x in {0,5}
	{\node[color = red] at (6+\x,0) {$\blacklozenge$};}
\foreach \x in {1}
	{\node[color = red] at (6+\x,1) {$\blacklozenge$};}

\end{tikzpicture}
\end{center}
\begin{center}
\begin{tikzpicture}[scale = 1.25]
\draw [line width = 0.3mm, line cap = round] (0,0)--node[left]{0'}(0,1);
\draw [line width = 0.3mm, line cap = round] (0,1)--node[above]{1'}(1,1);
\draw [line width = 0.3mm, line cap = round] (1,1)--node[above]{2}(2,1);
\draw [line width = 0.3mm, line cap = round] (2,1)--node[above]{$\cdots$}(3,1);
\draw [line width = 0.3mm, line cap = round] (3,1)--node[above]{$\cdots$}(4,1);
\draw [line width = 0.3mm, line cap = round] (4,1)--node[above]{$N$}(5,1);
\draw [line width = 0.3mm, line cap = round] (5,0)--node[below]{$\cdots$}(4,0);
\draw [line width = 0.3mm, line cap = round] (4,0)--node[below]{$\cdots$}(3,0);
\draw [line width = 0.3mm, line cap = round] (3,0)--node[below]{$\cdots$}(2,0);
\draw [line width = 0.3mm, line cap = round] (2,0)--node[below]{$\cdots$}(1,0);
\draw [line width = 0.3mm, line cap = round] (1,0)--node[below]{2}(0,0);
\draw [line width = 0.3mm, line cap = round] (0+5,1)--node[above]{1}(1+5,1);
\draw [line width = 0.3mm, line cap = round] (1+5,1)--node[above]{2'}(2+5,1);
\draw [line width = 0.3mm, line cap = round] (2+5,1)--node[above]{$\cdots$}(3+5,1);
\draw [line width = 0.3mm, line cap = round] (3+5,1)--node[above]{$\cdots$}(4+5,1);
\draw [line width = 0.3mm, line cap = round] (4+5,1)--node[above]{$M'$}(5+5,1);
\draw [line width = 0.3mm, line cap = round] (5+5,1)--node[right]{0'}(5+5,0);
\draw [line width = 0.3mm, line cap = round] (5+5,0)--node[below]{$\cdots$}(4+5,0);
\draw [line width = 0.3mm, line cap = round] (4+5,0)--node[below]{$\cdots$}(3+5,0);
\draw [line width = 0.3mm, line cap = round] (3+5,0)--node[below]{$\cdots$}(2+5,0);
\draw [line width = 0.3mm, line cap = round] (2+5,0)--node[below]{$\cdots$}(1+5,0);
\draw [line width = 0.3mm, line cap = round] (1+5,0)--node[below]{2'}(0+5,0);
\draw [densely dashed](1,0)--(1,1);
\draw [densely dashed](2,0)--(2,1);
\draw [densely dashed](3,0)--(3,1);
\draw [densely dashed](4,0)--(4,1);
\draw [densely dashed](5,0)--(5,1);
\draw [densely dashed](1+5,0)--(1+5,1);
\draw [densely dashed](2+5,0)--(2+5,1);
\draw [densely dashed](3+5,0)--(3+5,1);
\draw [densely dashed](4+5,0)--(4+5,1);
\draw [line width = 0.3mm, color = red] (0,0.5)--(10,0.5);
\draw [line width = 0.3mm, color = blue] (0.5,0)--(0.5,1);
\draw [line width = 0.3mm, color = blue] (1.5,0)--(1.5,1);
\draw [line width = 0.3mm, color = blue] (2.5,0)--(2.5,1);
\draw [line width = 0.3mm, color = blue] (3.5,0)--(3.5,1);
\draw [line width = 0.3mm, color = blue] (4.5,0)--(4.5,1);
\draw [line width = 0.3mm, color = blue] (5.5,0)--(5.5,1);
\draw [line width = 0.3mm, color = blue] (6.5,0)--(6.5,1);
\draw [line width = 0.3mm, color = blue] (7.5,0)--(7.5,1);
\draw [line width = 0.3mm, color = blue] (8.5,0)--(8.5,1);
\draw [line width = 0.3mm, color = blue] (9.5,0)--(9.5,1);
\draw [line width = 0.3mm, color = green] (0,0)--(1,1);
\draw [line width = 0.3mm, color = green] (5,0)--(6,1);
\foreach \x in {5}
	{\node[color = blue] at (\x,0) {$\bullet$};}
\foreach \x in {-6,4}
	{\node[color = red] at (6+\x,0) {$\blacklozenge$};}
\foreach \x in {6}
	{\node[color = blue] at (\x,1) {$\bullet$};}
\foreach \x in {-5}
	{\node[color = red] at (6+\x,1) {$\blacklozenge$};}
\end{tikzpicture}
\end{center}
\caption{Realisation of the surfaces $S_{1}$, $S_{2}$ and $S$.}
\label{GL1}
\end{figure}

Note that the surface produced by this method has the necessary form to be a constituent surface; that is, the bottom of the first square is again identified with the top of the second. As such, this process can be iterated.

Dealing only with the information about the strata of the surfaces being used, Lemma~\ref{comb1} has the following analogue in terms of the algebraic description of square-tiled surfaces: Let $(h_{1},v_{1})$ and $(h_{2},v_{2})$ be 1,1-square-tiled surfaces in $\calH(k_{1},\ldots,k_{n})$ and $\calH(l_{1},\ldots,l_{m})$, respectively, and with $h_{1} = (1,2,\ldots,N)$, $h_{2} = (1',2',\ldots,M')$, $v_{1}^{-1}(1) = 2$, and $v_{2}^{-1}(1') = 2'$. Then the square-tiled surface surface $(h,v)$ with $h = (1,1')h_{1}h_{2}$ and $v = v_{1}v_{2}(1,1')$ represents a 1,1-square-tiled surface with $[h,v] = ([h_{1},v_{1}][h_{2},v_{2}])^{(1,1')}$ and so in particular lies in $\calH(k_{1},\ldots,k_{n},l_{1},\ldots,l_{m})$.

We now consider how the spin structures of 1,1-square-tiled surfaces behave under cylinder concatenation. Indeed, this is the content of the following lemma.

\begin{lemma}\label{comb2}
Let $S_{1}\in\calH_{g_{1}}(2k_{1},\ldots,2k_{n})$ and $S_{2}\in\calH_{g_{2}}(2l_{1},\ldots,2l_{m})$ be 1,1-square-tiled surfaces with the form necessary to apply Lemma~\ref{comb1}. Further assume that $S_{1}$ has spin parity $\epsilon$, and $S_{2}$ has spin parity $\eta$. Let $S\in\calH_{g_{1}+g_{2}-1}(2k_{1},\ldots,2k_{n},2l_{1},\ldots,2l_{m})$ be the 1,1-square-tiled surface obtained from $S_{1}$ and $S_{2}$ by applying Lemma~\ref{comb1}, then $S$ has spin parity
\[\epsilon+\eta+1\mod 2.\]
\end{lemma}

\begin{proof}
Consider $S_{1}$ and $S_{2}$ as in Figure~\ref{GL1}. In each surface, we can choose the core curves of the horizontal cylinders and the green curves of slope 1 to form symplectic pairs $\{\alpha_{1},\beta_{1}\}$ and $\{\gamma_{1},\delta_{1}\}$, respectively. Both curves in each pair have constant angle with respect to the horizontal direction and so both have index 0. As such, we have
\[(\ind(\alpha_{1})+1)(\ind(\beta_{1})+1) = 1 = (\ind(\gamma_{1})+1)(\ind(\delta_{1})+1).\]
If the sets of curves $\{\alpha_{2},\beta_{2},\ldots,\alpha_{g_{1}},\beta_{g_{1}}\}$ and $\{\gamma_{2},\delta_{2},\ldots,\gamma_{g_{2}},\delta_{g_{2}}\}$ complete a symplectic basis on each surface, then we must have
\[\sum_{i = 2}^{g_{1}}(\ind(\alpha_{i})+1)(\ind(\beta_{i})+1) \equiv \epsilon  - 1 \mod 2,\]
and
\[\sum_{i = 2}^{g_{2}}(\ind(\gamma_{i})+1)(\ind(\delta_{i})+1) \equiv \eta  - 1 \mod 2.\]

In $S$, we can again choose the horizontal core curve and the green curve of slope 1 to form a symplectic pair $\{\mu,\nu\}$ satisfying $(\ind(\mu)+1)(\ind(\nu)+1) = 1$. A symplectic basis can then be completed by further taking the union of $\{\alpha_{2},\beta_{2},\ldots,\alpha_{g_{1}},\beta_{g_{1}}\}$ and $\{\gamma_{2},\delta_{2},\ldots,\gamma_{g_{2}},\delta_{g_{2}}\}$. Each curve will have the same index on $S$ as it did on $S_{1}$ or $S_{2}$, respectively. Hence, we see that the spin parity of $S$ is
\[1 + (\epsilon - 1) + (\eta - 1) \equiv \epsilon+\eta+1\mod 2,\]
as claimed.
\end{proof}

Unlike Lemma~\ref{comb1}, this lemma does not have a form that is easily stated in terms of the algebraic description of a square-tiled surface.

Returning to the example we gave above, one can check that the permutation we combined represented a surface in $\odd(4)$ and that the resulting permutation represents a surface in $\odd(4,4)$, as we would expect from Lemma~\ref{comb2}. Constructions like this will allow us to use a number of constituent surfaces to build 1,1-square-tiled surfaces in the desired connected components of general strata.


\section{Construction of 1,1-square-tiled surfaces}\label{con}

In this section we construct 1,1-square-tiled surfaces in every connected component of every stratum of Abelian differentials using the minimum number of squares possible and hence proving Theorem~\ref{STS}. Though most results are stated in terms of permutation representatives, the proofs will make use of the filling pair diagrams introduced in Subsection~\ref{fpd}. The construction relies heavily on the combination lemmas (Lemmas~\ref{comb1} and~\ref{comb2}).

\subsection{Outline of proof}

We begin in Subsection~\ref{hyp} by constructing by hand 1,1-square-tiled surfaces in the hyperelliptic components of strata. For a non-hyperelliptic component in an arbitrary stratum $\calH(k_{1},\ldots,k_{n})$ we will employ Lemmas~\ref{comb1} and~\ref{comb2}. That is, for every even $k_{i}$ we will construct a 1,1-square-tiled surface in $\calH(k_{i})$, and for every pair of odd $\{k_{i},k_{j}\}$ we will construct a 1,1-square-tiled surface in $\calH(k_{i},k_{j})$. Then a 1,1-square-tiled surface in $\calH(k_{1},\ldots,k_{n})$ can be constructed by inductively applying Lemma~\ref{comb1} to these surfaces. Moreover, if we can construct 1,1-square-tiled surfaces in the odd and even components of $\calH(2k)$, then by inductively applying Lemma~\ref{comb2}, we can build a 1,1-square-tiled surface in the odd and even components of $\calH(2k_{1},\ldots,2k_{n})$. The construction of 1,1-square-tiled surfaces in $\odd(2k)$ and $\even(2k)$ is carried out in Subsection~\ref{Even}. The construction of 1,1-square-tiled surfaces in $\calH(2j+1,2k+1)$ is carried out in Subsection~\ref{Odd}. Finally, the construction of 1,1-square-tiled surfaces in certain strata with a mix of odd and even order zeros is carried out in Subsection~\ref{Gen}.

Unfortunately, this method is complicated by the strata $\calH(2)$, and $\calH(1,1)$ for which there do not exist 1,1-square-tiled surfaces built from the theoretical minimum number of squares. See Proposition~\ref{hypmin}. Moreover, there do not exist $\even(4)$ and $\even(2,2)$ components that can be used in our construction. As such, we are required to modify the ideal method described above in such situations.

\subsection{Hyperelliptic components}\label{hyp}

We begin by constructing 1,1-square-tiled surfaces in the hyperelliptic components that realise the number of squares claimed in Theorem~\ref{STS}. Indeed, this is the content of the following proposition. We will then prove that these are the minimum number of squares necessary for 1,1-square-tiled surfaces in the hyperelliptic components. The fact that these numbers are strictly greater than the minimum required for square-tiled surfaces in the ambient stratum, particularly for genus two, will cause us difficulty in the subsections that follow. Indeed, since the strata $\calH(2)$ and $\calH(1,1)$ are connected and coincide with their hyperelliptic components, we will not have 1,1-square-tiled surfaces in these strata that can be used to build minimal 1,1-square-tiled surfaces in higher genus strata.

\begin{proposition}\label{Hyp}
For $g\geq 2$, the permutations
\begin{equation}\label{hyp2g-2}
\begin{multlined}
\left(\begin{matrix}
0&1&2&3&4&\cdots&2g-5&2g-4&2g-3 \\
4g-4&4g-6&4g-5&4g-8&4g-7&\cdots&2g&2g+1&2g-1
\end{matrix}\color{white} \right) \\
\hspace{2.2cm}\color{white}\left( \color{black}
\begin{matrix}
2g-2&2g-1&\cdots&4g-8&4g-7&4g-6&4g-5&4g-4\\
2g-3&2g-2&\cdots&3&4&1&2&0
\end{matrix}\right)
\end{multlined}
\end{equation}
and 
\begin{equation}\label{hypg-1}
\begin{multlined}
\left(\begin{matrix}
0&1&2&3&4&\cdots&2g-3&2g-2&2g-1 \\
4g-2&4g-4&4g-3&4g-6&4g-5&\cdots&2g&2g+1&2g-1
\end{matrix}\color{white} \right) \\
\hspace{2.2cm}\color{white}\left( \color{black}
\begin{matrix}
2g&2g+1&\cdots&4g-6&4g-5&4g-4&4g-3&4g-2\\
2g-3&2g-2&\cdots&3&4&1&2&0
\end{matrix}\right)
\end{multlined}
\end{equation}
represent 1,1-square-tiled surfaces in $\hyp(2g-2)$ and $\hyp(g-1,g-1)$, respectively.
\end{proposition}

\begin{proof}
We first note that permutations (\ref{hyp2g-2}) and (\ref{hypg-1}) are produced by adding $2g-3$ and $2g-2$ marked points to the standard permutations for $\hyp(2g-2)$ and $\hyp(g-1,g-1)$, respectively. Essentially originally due to Veech \cite{V} and stated explicitly by Zorich {\cite[Proposition 6]{Z}}, up to a relabelling, these are
\[\left(\begin{matrix}
0&1&\cdots&2g-2&2g-1\\
2g-1&2g-2&\cdots&1&0
\end{matrix}\right)\text{ and }
\left(\begin{matrix}
0&1&\cdots&2g-1&2g \\
2g&2g-1&\cdots&1&0
\end{matrix}\right).\]
Indeed, in the case of $\hyp(2g-2)$, one can check that we have added marked points by splitting the sides labelled by $i$ for $1\leq i\leq g-1$, $g+1\leq i\leq 2g-2$. Similarly, for $\hyp(g-1,g-1)$ we have split the sides labelled by $i$ for $1\leq i\leq g-1$, $g+1\leq i\leq 2g-1$. As such, permutations (\ref{hyp2g-2}) and (\ref{hypg-1}) do indeed represent $\hyp(2g-2)$ and $\hyp(g-1,g-1)$, respectively. Therefore, we need only prove that the permutations have one vertical and one horizontal cylinder when representing a square-tiled surface.

It is clear that these permutation representatives give rise to square-tiled surfaces with one horizontal cylinder. To check for the vertical cylinder we look at the permutation obtained by removing the 0s from the permutation representative. For permutation (\ref{hyp2g-2}), the permutation with the 0s removed is the following cycle:
\[(1,4g-4,2,4g-6,4,\ldots,2g,2g-2,2g-1,2g-3,2g+1,2g-5,2g+3,\ldots,4g-7,3,4g-5).\]
We have a single cycle and so we see that the square-tiled surface has one vertical cylinder. Hence we do indeed have a 1,1-square-tiled surface. Similarly, for permutation (\ref{hypg-1}), the permutation with the 0s removed is a cycle as follows:
\[(1,4g-2,2,4g-4,4,\ldots,2g-2,2g,2g-1,2g+1,2g-3,2g+3,2g-5,\ldots,4g-5,3,4g-3).\]
Again, since we have a single cycle, we see that the square-tiled surface has one vertical cylinder and so we have a 1,1-square-tiled surface. Hence, the proposition is proved.
\end{proof}

We observe that the 1,1-square-tiled surfaces built from these permutation representatives exhibit the number of squares claimed in the statement of Theorem~\ref{STS}; that is, we have $4g-4$ squares for $\hyp(2g-2)$ and $4g-2$ squares for $\hyp(g-1,g-1)$. To finish the proof of Theorem~\ref{STS} for the hyperelliptic cases we must show that these are in fact the minimum number of squares required for these components.

\begin{proposition}\label{hypmin}
A 1,1-square-tiled surface in the stratum $\hyp(2g-2)$ or $\hyp(g-1,g-1)$ requires at least $4g-4$ or $4g-2$ squares, respectively.
\end{proposition}

\begin{proof}
We formalise and generalise a method for genus two surfaces attributed to Margalit in a remark in a paper of Aougab-Huang in which they determine the minimal geometric intersection numbers for filling pairs on closed surfaces {\cite[Remark 2.18]{AH}}. The idea is to investigate the combinatorics of the images of the filling pair under the quotient by the hyperelliptic involution. If there is an arc between two punctures on the quotient sphere that is disjoint from the images of the filling pair, then this arc lifts to a curve disjoint from the filling pair on the original surface which contradicts the fact that the curves were filling.

Suppose that we have a 1,1-square-tiled surface $(S,\omega)$ in $\hyp(2g-2)$ made from $n$ squares and assume $n$ to be minimal. The core curves, $\alpha$ and $\beta$, of the vertical and horizontal cylinders form a filling pair of curves on the surface with geometric intersection number equal to $n$. Every intersection occurs with the same orientation, and so $\alpha$ and $\beta$ are nonseparating. Since $S$ is hyperelliptic, there exists an isometric involution $\tau:\!S\to S$ and a branched double cover $\pi:\! S\to S_{0,2g+2}$ of the sphere with $2g+2$ punctures. Since $\tau^{*}\omega = -\omega$, the vertical and horizontal cylinders are sent to vertical and horizontal cylinders, respectively. Moreover, since $\tau$ acts by isometry, the number of such cylinders is fixed. Hence, $\alpha$ and $\beta$ are nonseparating curves fixed by the hyperelliptic involution and so we have that $\pi(\alpha)$ and $\pi(\beta)$ are simple arcs on $S_{0,2g+2}$.

If $n$ is odd then, since any interior intersection of the arcs $\pi(\alpha)$ and $\pi(\beta)$ will lift to two intersections of $\alpha$ and $\beta$ on $S$, $\pi(\alpha)$ and $\pi(\beta)$ must share a single endpoint at a puncture on the sphere and have $(n-1)/2$ interior intersections. The arcs form a graph on the sphere with $3 + (n-1)/2$ vertices. Apart from the endpoints of the arcs which have valency 1 or 2, each vertex has valency 4, and so we have $n+1$ edges. It follows from an Euler characteristic argument that the resulting graph has $(n+1)/2$ complimentary regions. As mentioned above, we must have a maximum of one puncture in each complementary region. Three of the punctures lie at the endpoints of the arcs and so we must have
\[\frac{n+1}{2}\geq 2g-1 \Rightarrow n\geq 4g-3.\]

If $n$ is even then, by a similar argument to that given for $n$ odd above, $\pi(\alpha)$ and $\pi(\beta)$ either share both of their endpoints or have disjoint endpoints. In the former case, we have $(n-2)/2$ interior intersections. The arcs form a graph with $2+(n-2)/2$ vertices and $n$ edges. Hence we have $(n+2)/2$ complimentary regions. Two of the punctures lie at the endpoints and so we must have
\[\frac{n+2}{2}\geq 2g \Rightarrow n\geq 4g-2.\]
In the latter case, we have $n/2$ interior intersections. The arcs form a graph with $4+n/2$ vertices and $n+2$ edges. Hence we have $n/2$ complimentary regions. Four punctures lie at endpoints and so we must have
\[\frac{n}{2} \geq 2g-2 \Rightarrow n\geq 4g-4.\]
Hence we see that a 1,1-square-tiled surface in $\hyp(2g-2)$ requires at least $4g-4$ squares.

Suppose now that $S$ is a 1,1-square-tiled surface in $\hyp(g-1,g-1)$ with $n$ squares, with $n$ again assumed to be minimal. As above, the core curves of the cylinders, $\alpha$ and $\beta$, are nonseparating curves with geometric intersection number equal to $n$ and are fixed by the hyperelliptic involution. Hence, we have that $\pi(\alpha)$ and $\pi(\beta)$ are simple arcs on $S_{0,2g+2}$. Moreover, we must again have a maximum of one puncture in each complimentary region of the arcs. However, since the zeros of $\omega$ are by definition symmetric to one another by the hyperelliptic involution, they will correspond to a complementary region between $\pi(\alpha)$ and $\pi(\beta)$ that does not contain a puncture. So, in this case, we require one more complementary region between the arcs than we needed for $\hyp(2g-2)$. As such, we require an additional interior intersection of the arcs, which corresponds to two additional intersections of the filling pair, and so we must have $n\geq 4g-2$. This completes the proof of the proposition.
\end{proof}

\subsection{Even order zeros}\label{Even}

In this subsection we will construct 1,1-square-tiled surfaces in the odd and even components of strata with even order zeros. That, is we construct 1,1-square-tiled surfaces with a minimal number of squares in the odd and even components of strata of the form $\calH(2k_{1},\ldots,2k_{n})$, $k_{i}\geq 1$ and $\sum k_{i}=2g-2$. We do this by building base cases in the odd and even components of the strata $\calH(2k)$, $k\geq 2$. These surfaces can then be combined using the cylinder concatenation methods of Lemmas~\ref{comb1} and~\ref{comb2}. We must also deal with the fact that 1,1-square-tiled surfaces in $\calH(2)$ cannot be used to construct 1,1-square-tiled surfaces in higher genus strata. Moreover, there are no components $\even(4)$ and $\even(2,2)$ and so we have more work to do in order to be able to construct 1,1-square-tiled surfaces in the even components of strata having only zeros of order 2 or 4. \\

\paragraph*{{\bf Strata of the form $\bs{\calH(2k)}$}}

We begin by constructing 1,1-square-tiled surfaces in $\odd(2k)$, for $k\geq 2$.

\begin{proposition}\label{Oddmin}
The permutations
\begin{equation}\label{odd4}
\left(\begin{matrix}
0&1&2&3&4&5 \\
2&5&4&1&3&0
\end{matrix}\right),
\end{equation}
and, for  $g\geq 4$,
\begin{equation}\label{oddmin}
\left(\begin{matrix}
0&1&2&3&4&5&6&7&8&9&\cdots&2g-4&2g-3&2g-2&2g-1\\
2&5&4&7&3&9&6&11&8&13&\cdots&2g-4&1&2g-2&0
\end{matrix}\right)
\end{equation}
represent 1,1-square-tiled surfaces in $\odd(4)$ and $\odd(2g-2)$, respectively. Moreover, these surfaces have the minimum number of squares necessary for their respective strata.
\end{proposition}

\begin{proof}
We will first prove that the permutations represent the claimed strata. Observe that the filling pair in Figure~\ref{odd4fig} represents permutation (\ref{odd4}). We observe that its ribbon graph has one boundary component with 20 sides and so corresponds to a single zero of order 4. That is, the permutation represents a 1,1-square-tiled surface in $\calH(4)$.

Now consider the filling pair diagram in Figure~\ref{oddmin1} with $2g-1$ vertices. We will modify this diagram to produce a filling pair diagram representing permutation (\ref{oddmin}).

We will perform a series of vertex transpositions on the filling pair diagram. These will not change the fact that we have one vertical and one horizontal cylinder but will change the number of boundary components and the number of sides of the boundary components of the associated ribbon graph. We currently have $2g-1$ boundary components with four sides. Our goal is to produce a filling pair diagram with a single boundary component with $8g-4$ sides.

We first perform two transpositions on the third, fourth and fifth vertices to give the permutation $(3,5,4)$ on the vertices. Note that this gives the first 5 vertices the combinatorics given by the filling pair diagram for $\calH(4)$ in Figure~\ref{odd4fig}. Moreover, we now have one boundary component with 20 sides and $2g-6$ boundary components with 4 sides. See Figure~\ref{oddmin2}.

\begin{figure}[H]
\begin{center}
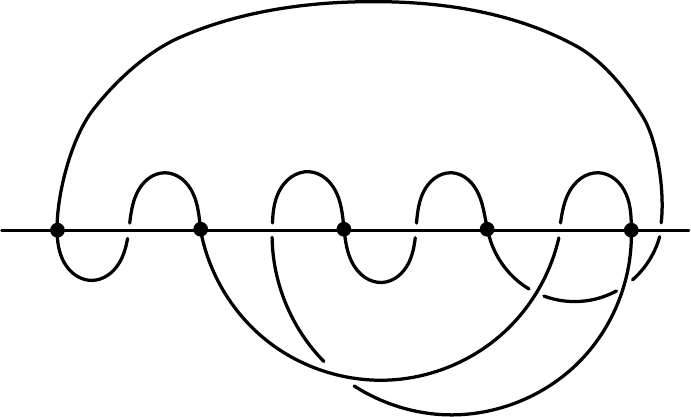
\end{center}
\caption{A filling pair diagram representing permutation (\ref{odd4}).}\label{odd4fig}
\end{figure}

\begin{figure}[H]
\begin{center}
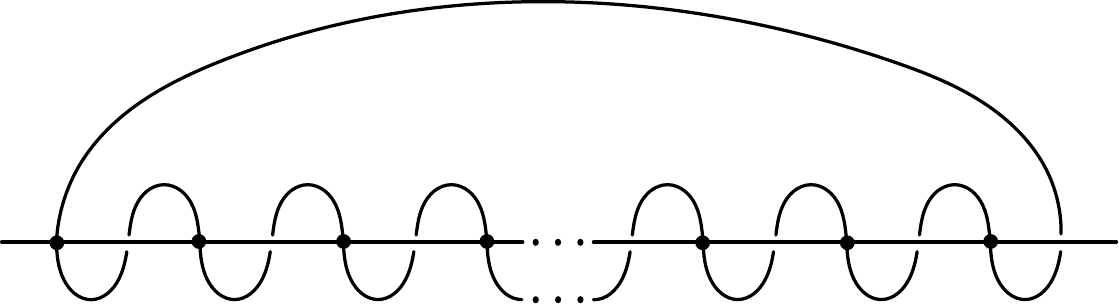
\end{center}
\caption[A filling pair diagram with $2g-1$ vertices.]{A filling pair diagram with $2g-1$ vertices. The associated ribbon graph has $2g-1$ boundary components each with 4 sides.}
\label{oddmin1}
\end{figure}

\begin{figure}[H]
\begin{center}
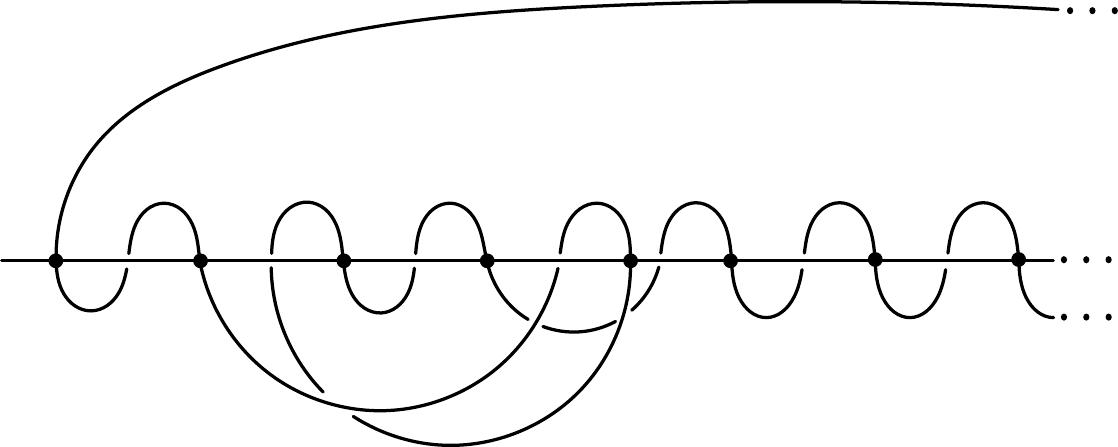
\end{center}
\caption[The filling pair diagram after applying permutation $(3,5,4)$.]{The filling pair diagram after applying permutation $(3,5,4)$ on the vertices.}
\label{oddmin2}
\end{figure}

Observe that after this permutation the boundary components around vertices 6-9 have the combinatorics shown on the left of Figure~\ref{oddmin3}, where different letters correspond to different boundary components. We then perform a vertex transposition on vertices 6 and 7, as shown on the right of Figure~\ref{oddmin3}. We now have a boundary component with 28 sides and $2g-8$ boundary components with 4 sides.

We now observe that, after the vertex transposition, the combinatorics that we had around vertices 6-9 are repeated again around vertices 8-11. As such, we can perform this transposition again to produce a boundary component with 36 sides and $2g-10$ boundary components with 4 sides. Moreover, these combinatorics persist and so we can continue to repeat this transposition for the remaining $g-5$ pairs of vertices ending up with a single boundary component. Since each vertex has valency 4, there are $4g-2$ edges in the filling pair diagram. Each edge will give two sides to the boundary component and so the boundary component will have $8g-4$ sides corresponding to a zero of order $2g-2$, as required.

It is easy to check that this filling pair diagram represents permutation (\ref{oddmin}), and so we have shown that this permutation does indeed represent a 1,1-square-tiled surface in $\calH(2g-2)$.

\begin{figure}[H]
\begin{center}
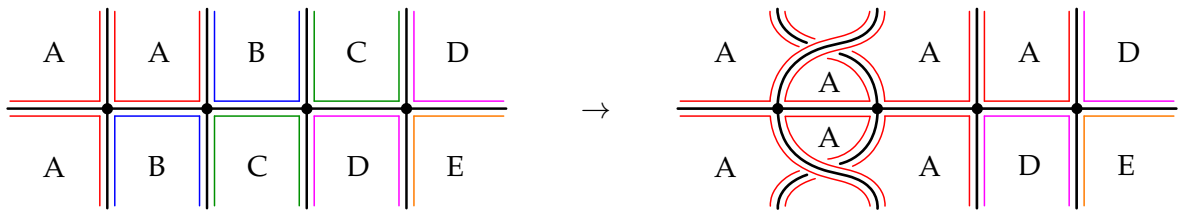
\end{center}
\caption[The effect of a vertex transposition on the boundary components.]{The effect of a vertex transposition on the boundary components of the ribbon graph of the filling pair diagram in Figure~\ref{oddmin2}.}
\label{oddmin3}
\end{figure}

\begin{figure}[H]
\begin{center}
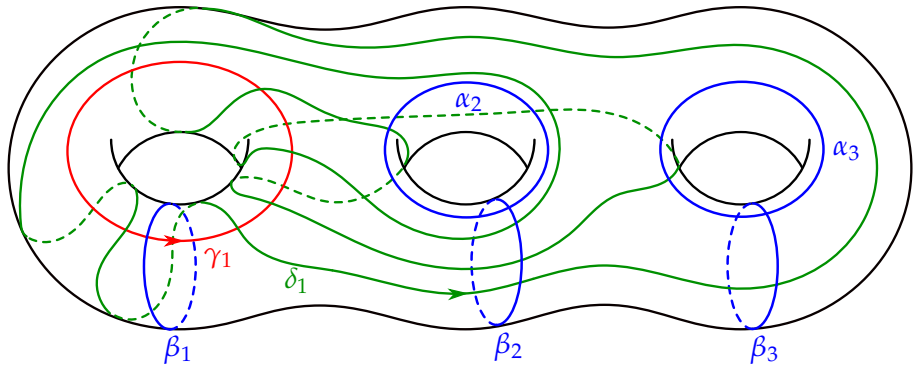
\end{center}
\caption{Realisation of the filling pair diagram representing permutation (\ref{odd4}).}
\label{oddspin1}
\end{figure}

\begin{figure}[H]
\begin{center}
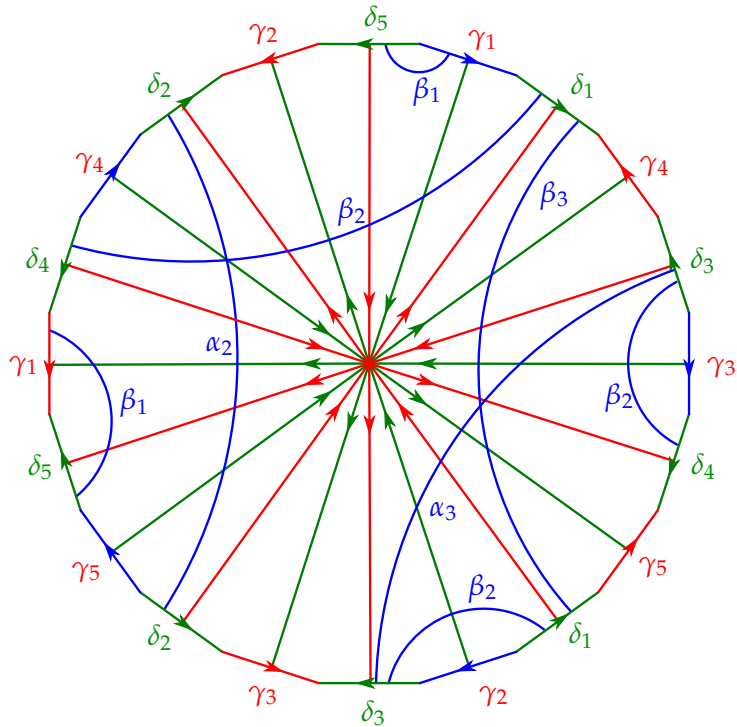
\end{center}
\caption{Polygonal decomposition of the surface given by the filling pair $(\gamma,\delta)$.}
\label{oddspin2}
\end{figure}

We must now show that these 1,1-square-tiled surfaces lie in the odd components. We will do this by calculating the spin parity of the surfaces with respect to representatives corresponding to the standard homology basis.

We first prove that the surface represented by permutation (\ref{odd4}) has odd spin structure and thus represents $\odd(4)$. We realise the filling pair diagram as the curves $\gamma$ and $\delta$ in Figure~\ref{oddspin1}, and label the arcs of each curve between their intersections with the labels $\gamma_{1},\ldots,\gamma_{5}$, and $\delta_{1},\ldots,\delta_{5}$, respectively. Here $\gamma$ corresponds to the horizontal core curve and $\delta$ to the vertical. Note that we only show the first label of each curve in the diagram. Next, choose the homology representatives $\{\alpha_{i},\beta_{i}\}_{i=1}^{3}$ as in Figure~\ref{oddspin1}. We choose $\alpha_{1} = \gamma$ and $\beta_{1}$ to be the curve of slope 1 with respect to the horizontal direction.

We now cut the surface open along the filling pair $\{\gamma,\delta\}$ to form the 20-gon shown in Figure~\ref{oddspin2}. We have also included in the figure the leaves of the vertical and horizontal foliations given by the edges of the squares making up the surface. The index of a curve can then be calculated by keeping track of the number of these lines the curve crosses and in which direction. It is then easy to show that we have
\[\sum_{i=1}^{3}(\ind(\alpha_{i})+1)(\ind(\beta_{i})+1)\equiv 1\text{ mod }2,\]
and so the canonical spin structure on the surface has odd spin parity.

\begin{figure}[H]
\begin{center}
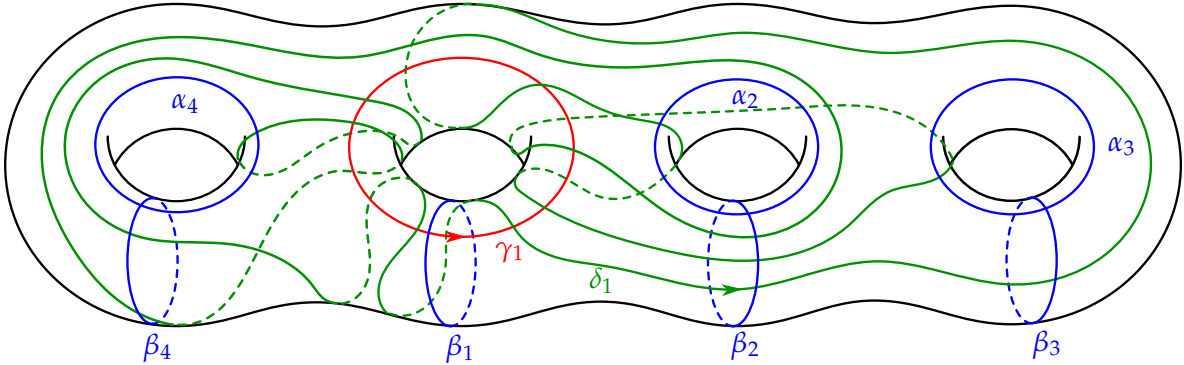
\end{center}
\caption{Realisation of filling pair diagram for $\calH(6)$.}
\label{oddspin3}
\end{figure}

The filling pair diagram given by permutation (\ref{oddmin}) representing $\calH(6)$ can be realised as in Figure~\ref{oddspin3}. A similar calculation to that above shows that the canonical spin structure on this surface also has odd spin parity. 

From this point onwards, for every additional increase of $g$ in permutation (\ref{oddmin}), the polygonal decompositions given by the realisations of the filling pair diagrams vary in a predictable manner. This is because the `final handle' on the surface, the handle associated to the final two vertices of the filling pair diagram, has the form of the handle on the left of Figure~\ref{oddspin3}; that is, the handle containing $\alpha_{4}$ and $\beta_{4}$. The change to the polygonal decompositions is then demonstrated by the changes between Figures~\ref{oddspin4} and~\ref{oddspin5}.

We see that the standard homology representatives around the added genus, $\alpha_{g}$ and $\beta_{g}$, both have index 1 and so contribute 0 to the calculation of the spin parity modulo 2. Moreover, 8 sides are added to the polygon in one piece and so, since each additional side crossed requires a rotation by $\pi/2$, the index of a curve passing these sides will change by 2 and so the contribution to the calculation of the spin parity is changed by 0 modulo 2. Altogether, we have added 0 modulo 2 to the calculation of the spin parity and so, since the surface in Figure~\ref{oddspin3} had odd spin parity, the surface represented by permutation (\ref{oddmin}) lies in $\odd(2g-2)$.

Note also that all permutations in the proposition represent square-tiled surfaces with the minimum number of squares for the respective strata, namely $2g-1$. As such, the proposition has been proved.
\end{proof}

\begin{figure}[H]
\begin{center}
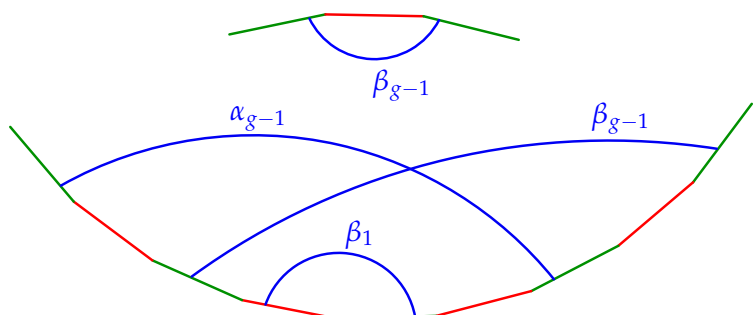
\end{center}
\caption[Part of the polygonal decomposition of surface of genus $g-1$.]{Part of the polygonal decomposition of surface of genus $g-1$ given by permutation (\ref{oddmin}).}
\label{oddspin4}
\end{figure}

\begin{figure}[H]
\begin{center}
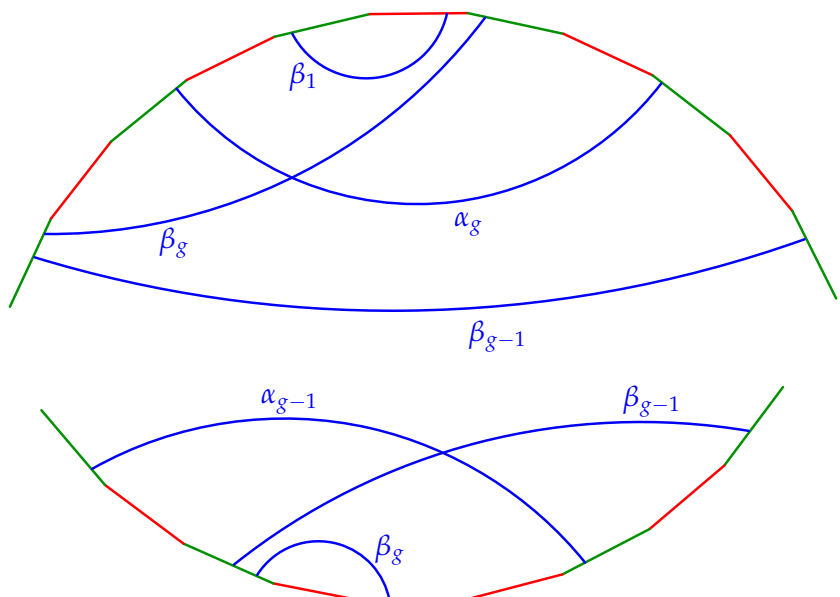
\end{center}
\caption[Part of the polygonal decomposition of surface of genus $g$.]{Part of the polygonal decomposition of surface of genus $g$ given by permutation (\ref{oddmin}).}
\label{oddspin5}
\end{figure}

Since the surfaces given by Proposition~\ref{Oddmin} have the desired form we can use Lemmas~\ref{comb1} and~\ref{comb2} to produce 1,1-square-tiled surfaces in the odd components of all strata with even order zeros of order greater than or equal to 4.

The following proposition constructs surfaces in the even components of the strata $\calH(2g-2)$.

\begin{proposition}\label{Evenmin}
The permutations
\begin{equation}\label{6even}
\left(\begin{matrix}
0&1&2&3&4&5&6&7 \\
2&7&6&5&3&1&4&0
\end{matrix}\right),
\end{equation}
and, for $g\geq 5$,
\begin{equation}\label{2g-2even}
\begin{multlined}
\left(\begin{matrix}0&1&2&3&4&5&6&7&8&9&10 \\
2&7&6&5&3&9&4&11&8&13&10
\end{matrix}\color{white} \right) \\
\hspace{2.2cm}\color{white}\left( \color{black}
\begin{matrix}
11&\cdots&2g-4&2g-3&2g-2&2g-1 \\
15&\cdots&2g-4&1&2g-2&0
\end{matrix}\right)
\end{multlined}
\end{equation}
represent 1,1-square-tiled surfaces in $\even(6)$ and $\even(2g-2)$, respectively. Moreover, these surfaces have the minimum number of squares necessary for their respective strata.
\end{proposition}

\begin{proof}
The proof is completely analogous to the proof of Proposition~\ref{Oddmin}. That is, one can show that permutation (\ref{6even}) represents a 1,1-square-tiled surface in $\even(6)$ by calculating directly on the polygonal decomposition given by the filling pair. Applying the same induction used in the proof of Proposition~\ref{Oddmin} then shows that permutation (\ref{2g-2even}) represents a 1,1-square-tiled surface in $\even(2g-2)$.
\end{proof}

\paragraph*{{\bf Handling the non-existence of $\bs{\even(4)}$}}

We now have the first instance in which we must construct an exceptional case separately. Note that using Lemmas~\ref{comb1} and~\ref{comb2}, we can use surfaces given by Propositions~\ref{Oddmin} and~\ref{Evenmin} to produce 1,1-square-tiled surfaces in the even components of all strata with even order zeros of order greater than or equal to 4 apart from strata containing only zeros of order 4. This is because there is no even component in the stratum $\calH(4)$. However, the permutation
\[\left(\begin{matrix}
0&1&2&3&4&5&6&7&8&9&10 \\
2&10&7&5&8&1&9&6&4&3&0
\end{matrix}\right)\]
represents a 1,1-square-tiled surface in $\even(4,4)$ and so we can use this to produce 1,1-square-tiled surfaces in the even components of these exceptional strata. \\

\paragraph*{{\bf Handling the hyperellipticity of $\bs{\calH(2)}$ and non-existence of $\bs{\even(2,2)}$}}

As we saw in the previous subsection, all surfaces in genus two are hyperelliptic and so 1,1-square-tiled surfaces require more than the minimum number of squares required for a square-tiled surface in their respective stratum. As such, we do not have a 1,1-square-tiled surface in $\calH(2)$ that can be concatenated, as in Lemmas~\ref{comb1} and~\ref{comb2}, with the surfaces we have produced in the propositions above. It is therefore necessary to produce strata containing zeros of order 2 separately. This is the content of the following propositions.

The first proposition produces 1,1-square-tiled surfaces in $\odd(2k,2)$, for $k\geq 3$.

\begin{proposition}\label{2Odd}
The permutations
\begin{equation}\label{2odd6}
\left(\begin{matrix}
0&1&2&3&4&5&6&7&8&9&10 \\
2&5&4&6&3&8&10&7&1&9&0
\end{matrix}\right),
\end{equation}
and, for $k\geq 4$,
\begin{equation}\label{2odd}
\begin{multlined}
\left(\begin{matrix}
0&1&2&3&4&5&6&7&8&9&\cdots&2k-4&2k-3 \\
2&5&4&7&3&9&6&11&8&13&\cdots&2k-4&2k
\end{matrix}\color{white}\right) \\
\hspace{2cm}\color{white}\left(\color{black}\begin{matrix}
2k-2&2k-1&2k&2k+1&2k+2&2k+3&2k+4 \\
2k-2&2k+2&2k+4&2k+1&1&2k+3&0
\end{matrix}\right)
\end{multlined}
\end{equation}
represent 1,1-square-tiled surfaces in $\odd(6,2)$ and $\odd(2k,2)$, respectively. Moreover, these surfaces have the minimum number of squares necessary for their respective strata.
\end{proposition}

\begin{proof}
We first begin by realising permutation (\ref{2odd6}) by the filling pair diagram in Figure~\ref{2odd1}. 

Note that to produce this filling pair diagram we have added the 5 vertices shown in Figure~\ref{f:H2} to the right-hand side of filling pair diagram for $\odd(4)$ in Figure~\ref{odd4fig}. These additional vertices contribute another 10 edges to the diagram and so another 20 sides to the boundary components of the associated ribbon graph. It is easy to check that there is a boundary component with 12 sides corresponding to a zero of order 2, and that the 8 remaining additional sides are added to the boundary component that represented the zero of order 4 in the original ribbon graph. Hence, we have a second boundary component corresponding to a zero of order 6. That is, the filling pair diagram represents a 1,1-square-tiled surface in the stratum $\calH(6,2)$.

The argument for permutation (\ref{2odd}) is similar in that we add 5 vertices in the same way to the right-hand side of the filling pair diagram representing permutation (\ref{oddmin}). These have a single zero of order 2, as above, and a second zero of order two more than the order of the zero represented by permutation (\ref{oddmin}). It is easy to check that the resulting 1,1-square-tiled surface lies in the claimed stratum.

To check that the surfaces have the claimed spin parity, we investigate the effect that modifying the permutations of Proposition~\ref{Oddmin} to achieve permutations (\ref{2odd6}) and (\ref{2odd}) has on the associated polygonal decompositions. This effect is demonstrated in Figure~\ref{2odd2}. We see that the 8 sides added to the polygon of the original surface are added in one piece and so, as was the case in Proposition~\ref{Oddmin}, the index of any curve crossing these sides is changed by 2 and so changes the calculation of the spin parity by 0 modulo 2. We also observe that the homology representatives, $\alpha_{g+1},\alpha_{g+2},\beta_{g+1}$, and $\beta_{g+2}$, around the two additional genus all have index 1 and so together contribute 0 modulo 2 to the calculation of the spin parity. Therefore, the spin parity of the resulting surface is the same as the spin parity of the surface we started with which in this case is odd. That is, permutations (\ref{2odd6}) and (\ref{2odd}) do indeed represent the odd components of their respective strata.

Finally, observe that the minimum number of squares required for square-tiled surfaces in $\calH(6,2)$ and $\calH(2k,2)$ are 10 and $2k+4$, respectively. As such, the 1,1-square-tiled surfaces we have produced have the minimum number of squares required for their respective strata. Hence the proposition is proved.
\end{proof}

\begin{figure}[H]
\begin{center}
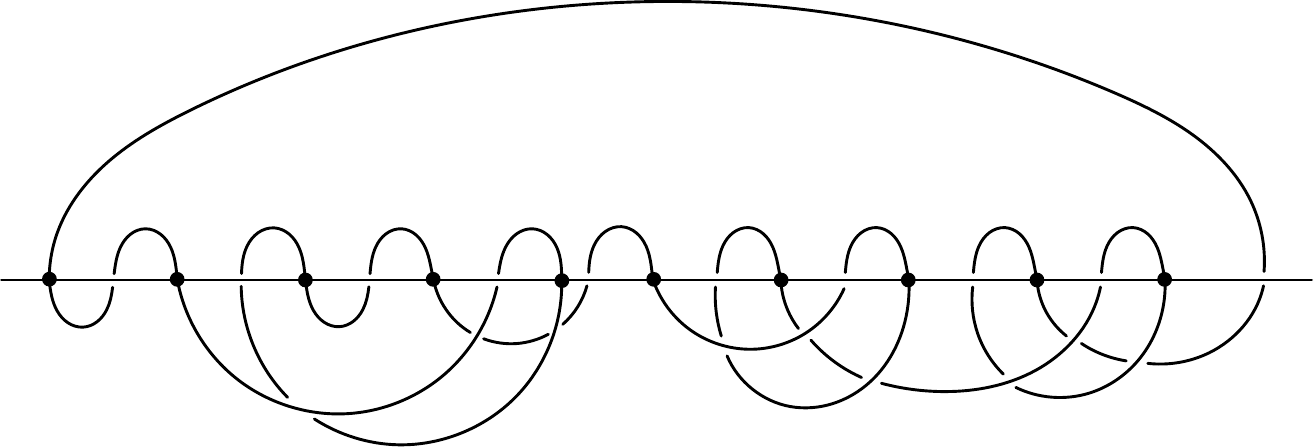
\end{center}
\caption{Filling pair diagram representing permutation (\ref{2odd6}).}
\label{2odd1}
\end{figure}

\begin{figure}[H]
\begin{center}
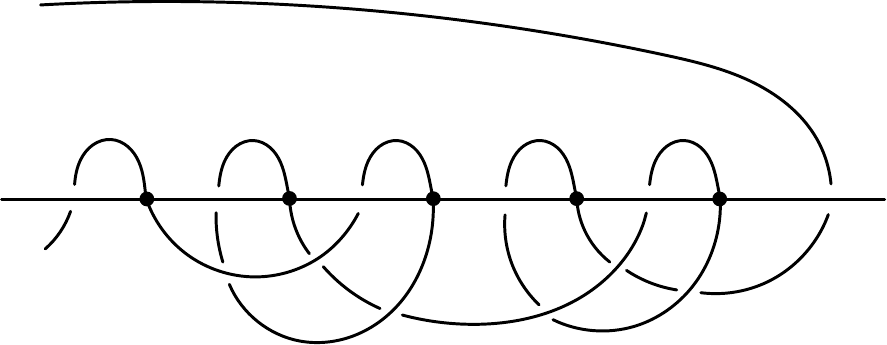
\end{center}
\caption{Filling pair diagram combinatorics for adding $\calH(2)$.}\label{f:H2}
\end{figure}

\begin{figure}[H]
\begin{center}
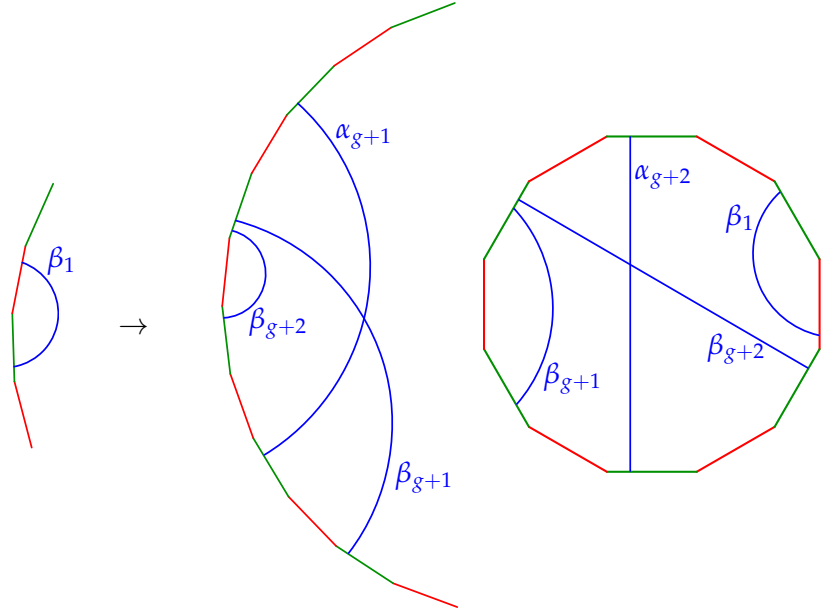
\end{center}
\caption[The effect on the polygonal decomposition.]{The effect on the polygonal decomposition of changing permutation (\ref{oddmin}) to permutation (\ref{2odd}), where the genus of the resulting surface is $g+2$.}
\label{2odd2}
\end{figure}

The following proposition produces 1,1-square-tiled surfaces in $\even(2k,2)$, for $k\geq 4$.

\begin{proposition}\label{2Even}
The permutations
\begin{equation}
\left(\begin{matrix}
0&1&2&3&4&5&6&7&8&9&10&11&12 \\
2&7&6&5&3&8&4&10&12&9&1&11&0
\end{matrix}\right),
\end{equation}
and, for $k\geq 5$,
\begin{equation}
\begin{multlined}
\left(\begin{matrix}
0&1&2&3&4&5&6&7&8&9&10&11&\cdots&2k-4&2k-3\\
2&7&6&5&3&9&4&11&8&13&10&15&\cdots&2k-4&2k
\end{matrix}\color{white}\right) \\
\hspace{2cm}\color{white}\left(\color{black}\begin{matrix}
2k-2&2k-1&2k&2k+1&2k+2&2k+3&2k+4 \\
2k-2&2k+2&2k+4&2k+1&1&2k+3&0
\end{matrix}\right)
\end{multlined}
\end{equation}
represent 1,1-square-tiled surfaces in $\even(8,2)$ and $\even(2k,2)$, respectively. Moreover, these surfaces have the minimum number of squares necessary for their respective strata.
\end{proposition}

\begin{proof}
The proof is analogous to the proof of Proposition~\ref{2Odd}. That is, we add the filling pair diagram combinatorics shown in Figure~\ref{f:H2} in the same way as above to the filling pair diagrams representing the permutations of Proposition~\ref{Evenmin}.
\end{proof}

As above, we have a number of exceptional cases not covered by these propositions. These are resolved as follows. The permutation
\[\left(\begin{matrix}
0&1&2&3&4&5&6 \\
2&4&6&3&1&5&0
\end{matrix}\right)\]
represents $\odd(2,2)$. The permutations
\[\left(\begin{matrix}
0&1&2&3&4&5&6&7&8&9 \\
2&8&6&9&4&1&3&5&7&0
\end{matrix}\right)\text{   and   }\left(\begin{matrix}
0&1&2&3&4&5&6&7&8&9 \\
2&9&8&7&6&3&5&1&4&0
\end{matrix}\right)\]
represent $\odd(2,2,2)$ and $\even(2,2,2)$, respectively. The permutation
\[\left(\begin{matrix}
0&1&2&3&4&5&6&7&8&9&10&11&12 \\
2&5&4&1&12&3&10&7&11&9&6&8&0
\end{matrix}\right)\]
represents $\even(2,2,2,2)$. The permutations
\[\left(\begin{matrix}
0&1&2&3&4&5&6&7&8 \\
2&5&8&3&6&4&1&7&0
\end{matrix}\right)\text{   and   }\left(\begin{matrix}
0&1&2&3&4&5&6&7&8 \\
2&4&1&8&7&5&3&6&0
\end{matrix}\right)\]
represent $\odd(4,2)$ and $\even(4,2)$, respectively. The permutation
\[\left(\begin{matrix}
0&1&2&3&4&5&6&7&8&9&10&11 \\
2&8&5&3&1&10&9&6&4&11&7&0
\end{matrix}\right)\]
represents $\even(4,2,2)$.
Finally, the permutation
\[\left(\begin{matrix}
0&1&2&3&4&5&6&7&8&9&10 \\
2&10&9&8&6&3&5&1&4&7&0
\end{matrix}\right)\]
represents $\even(6,2)$. It is easily checked that these permutation representatives do indeed represent 1,1-square-tiled surfaces in the required strata. That they represent surfaces in the correct connected component can be computed using the algorithm described by Zorich discussed in Subsection~\ref{permreps} or verified computationally. Using these surfaces and the those given by the propositions proved in this subsection, we can produce 1,1-square-tiled surfaces in the even and odd components of all strata of Abelian differentials that have even order zeros. This completes the work of this subsection.

\subsection{Odd order zeros}\label{Odd}

In this subsection, we will construct 1,1-square-tiled surfaces in all strata of Abelian differentials with odd order zeros. More specifically, if the stratum is not connected, we will construct them in the nonhyperelliptic component. To do this, we must construct the base cases $\calH(2j+1,2k+1)$. Similar to the difficulties caused by the hyperellipticity of $\calH(2)$, we must in this subsection also deal with the hyperellipticity of $\calH(1,1)$.

We remark that any 1,1-square-tiled surface lying in $\calH(g-1,g-1)$ constructed in this subsection will be contained in $\non(g-1,g-1)$ since it will be built from strictly fewer squares than that required by Proposition~\ref{hypmin} for a 1,1-square-tiled surface in $\hyp(g-1,g-1)$. \\

\paragraph*{{\bf Strata of the form $\bs{\calH(2j+1,2k+1)}$}}

Before giving the proofs of the constructions in this subsection, we give examples of the methods used in Propositions~\ref{RS} and~\ref{LS} below. We will see how we can use the 1,1-square-tiled surfaces representing $\odd(4)$ and $\odd(6)$ constructed in Proposition~\ref{Oddmin} to construct 1,1-square-tiled surfaces in $\calH(5,3)$, $\calH(5,5)$, $\calH(7,3)$ and $\calH(9,3)$.

To construct a 1,1-square-tiled surface in $\calH(5,3)$, we begin with two copies of the 1,1-square-tiled surface in $\odd(4)$ given by Proposition~\ref{Oddmin} as in the top of Figure~\ref{RSex}. We then perform cylinder concatenation on these two surfaces to obtain the 1,1-square-tiled surface shown in the middle of Figure~\ref{RSex}. Now consider in this surface the square (shaded grey) that lies to the right of the square that was the final square of the left-hand surface used in the cylinder concatenation. We swap this square with the one to its right while remembering the  vertical identifications and obtain the surface shown at the bottom of Figure~\ref{RSex}. The resulting surface is still a 1,1-square-tiled surface and it can be checked that this surface lies in $\calH(5,3)$. If we instead perform this operation with the left-hand surface being the 1,1-square-tiled surface in $\odd(6)$ given by Proposition~\ref{Oddmin} then it can be checked that the resulting surface lies in $\calH(5,5)$. We shall call such a procedure a {\em right-swap concatenation}. The investigation of the outcomes of this procedure in general is the content of Proposition~\ref{RS}.

\begin{center}
\begin{figure}[H]
\begin{center}
\begin{tikzpicture}[scale = 1.25]
\draw [line width = 0.3mm, line cap = round] (0,0)--node[left]{0}(0,1);
\draw [line width = 0.3mm, line cap = round] (0,1)--node[above]{1}(1,1);
\draw [line width = 0.3mm, line cap = round] (1,1)--node[above]{2}(2,1);
\draw [line width = 0.3mm, line cap = round] (2,1)--node[above]{3}(3,1);
\draw [line width = 0.3mm, line cap = round] (3,1)--node[above]{4}(4,1);
\draw [line width = 0.3mm, line cap = round] (4,1)--node[above]{5}(5,1);
\draw [line width = 0.3mm, line cap = round] (5,1)--node[right]{0}(5,0);
\draw [line width = 0.3mm, line cap = round] (5,0)--node[below]{3}(4,0);
\draw [line width = 0.3mm, line cap = round] (4,0)--node[below]{1}(3,0);
\draw [line width = 0.3mm, line cap = round] (3,0)--node[below]{4}(2,0);
\draw [line width = 0.3mm, line cap = round] (2,0)--node[below]{5}(1,0);
\draw [line width = 0.3mm, line cap = round] (1,0)--node[below]{2}(0,0);
\draw [densely dashed](1,0)--(1,1);
\draw [densely dashed](2,0)--(2,1);
\draw [densely dashed](3,0)--(3,1);
\draw [densely dashed](4,0)--(4,1);
\draw [line width = 0.3mm, color = red] (0,0.5)--(5,0.5);
\draw [line width = 0.3mm, color = blue] (0.5,0)--(0.5,1);
\draw [line width = 0.3mm, color = blue] (1.5,0)--(1.5,1);
\draw [line width = 0.3mm, color = blue] (2.5,0)--(2.5,1);
\draw [line width = 0.3mm, color = blue] (3.5,0)--(3.5,1);
\draw [line width = 0.3mm, color = blue] (4.5,0)--(4.5,1);

\draw [line width = 0.3mm, line cap = round] (0+6,0)--node[left]{0}(0+6,1);
\draw [line width = 0.3mm, line cap = round] (0+6,1)--node[above]{1}(1+6,1);
\draw [line width = 0.3mm, line cap = round] (1+6,1)--node[above]{2}(2+6,1);
\draw [line width = 0.3mm, line cap = round] (2+6,1)--node[above]{3}(3+6,1);
\draw [line width = 0.3mm, line cap = round] (3+6,1)--node[above]{4}(4+6,1);
\draw [line width = 0.3mm, line cap = round] (4+6,1)--node[above]{5}(5+6,1);
\draw [line width = 0.3mm, line cap = round] (5+6,1)--node[right]{0}(5+6,0);
\draw [line width = 0.3mm, line cap = round] (5+6,0)--node[below]{3}(4+6,0);
\draw [line width = 0.3mm, line cap = round] (4+6,0)--node[below]{1}(3+6,0);
\draw [line width = 0.3mm, line cap = round] (3+6,0)--node[below]{4}(2+6,0);
\draw [line width = 0.3mm, line cap = round] (2+6,0)--node[below]{5}(1+6,0);
\draw [line width = 0.3mm, line cap = round] (1+6,0)--node[below]{2}(0+6,0);
\draw [densely dashed](1+6,0)--(1+6,1);
\draw [densely dashed](2+6,0)--(2+6,1);
\draw [densely dashed](3+6,0)--(3+6,1);
\draw [densely dashed](4+6,0)--(4+6,1);
\draw [line width = 0.3mm, color = red] (6,0.5)--(11,0.5);
\draw [line width = 0.3mm, color = green] (6.5,0)--(6.5,1);
\draw [line width = 0.3mm, color = green] (7.5,0)--(7.5,1);
\draw [line width = 0.3mm, color = green] (8.5,0)--(8.5,1);
\draw [line width = 0.3mm, color = green] (9.5,0)--(9.5,1);
\draw [line width = 0.3mm, color = green] (10.5,0)--(10.5,1);
\end{tikzpicture}
\end{center}
\begin{center}
\begin{tikzpicture}[scale = 1.25]
\draw [line width = 0.3mm, line cap = round] (0,0)--node[left]{0}(0,1);
\draw [line width = 0.3mm, line cap = round] (0,1)--node[above]{1}(1,1);
\draw [line width = 0.3mm, line cap = round] (1,1)--node[above]{2}(2,1);
\draw [line width = 0.3mm, line cap = round] (2,1)--node[above]{3}(3,1);
\draw [line width = 0.3mm, line cap = round] (3,1)--node[above]{4}(4,1);
\draw [line width = 0.3mm, line cap = round] (4,1)--node[above]{5}(5,1);
\draw [line width = 0.3mm, line cap = round] (5,0)--node[below]{3}(4,0);
\draw [line width = 0.3mm, line cap = round] (4,0)--node[below]{6}(3,0);
\draw [line width = 0.3mm, line cap = round] (3,0)--node[below]{4}(2,0);
\draw [line width = 0.3mm, line cap = round] (2,0)--node[below]{5}(1,0);
\draw [line width = 0.3mm, line cap = round] (1,0)--node[below]{2}(0,0);
\draw [line width = 0.3mm, line cap = round] (0+5,1)--node[above]{6}(1+5,1);
\draw [line width = 0.3mm, line cap = round] (1+5,1)--node[above]{7}(2+5,1);
\draw [line width = 0.3mm, line cap = round] (2+5,1)--node[above]{8}(3+5,1);
\draw [line width = 0.3mm, line cap = round] (3+5,1)--node[above]{9}(4+5,1);
\draw [line width = 0.3mm, line cap = round] (4+5,1)--node[above]{10}(5+5,1);
\draw [line width = 0.3mm, line cap = round] (5+5,1)--node[right]{0}(5+5,0);
\draw [line width = 0.3mm, line cap = round] (5+5,0)--node[below]{8}(4+5,0);
\draw [line width = 0.3mm, line cap = round] (4+5,0)--node[below]{1}(3+5,0);
\draw [line width = 0.3mm, line cap = round] (3+5,0)--node[below]{9}(2+5,0);
\draw [line width = 0.3mm, line cap = round] (2+5,0)--node[below]{10}(1+5,0);
\draw [line width = 0.3mm, line cap = round] (1+5,0)--node[below]{7}(0+5,0);
\draw [densely dashed](1,0)--(1,1);
\draw [densely dashed](2,0)--(2,1);
\draw [densely dashed](3,0)--(3,1);
\draw [densely dashed](4,0)--(4,1);
\draw [densely dashed](5,0)--(5,1);
\draw [densely dashed](1+5,0)--(1+5,1);
\draw [densely dashed](2+5,0)--(2+5,1);
\draw [densely dashed](3+5,0)--(3+5,1);
\draw [densely dashed](4+5,0)--(4+5,1);
\draw [fill = gray, draw opacity = 0, fill opacity = 0.25] (5,0) -- (6,0) -- (6,1) -- (5,1) -- cycle;
\draw [line width = 0.3mm, color = red] (0,0.5)--(10,0.5);
\draw [line width = 0.3mm, color = blue] (0.5,0)--(0.5,1);
\draw [line width = 0.3mm, color = blue] (1.5,0)--(1.5,1);
\draw [line width = 0.3mm, color = blue] (2.5,0)--(2.5,1);
\draw [line width = 0.3mm, color = blue] (3.5,0)--(3.5,1);
\draw [line width = 0.3mm, color = blue] (4.5,0)--(4.5,1);
\draw [line width = 0.3mm, color = blue] (5.5,0)--(5.5,1);
\draw [line width = 0.3mm, color = blue] (6.5,0)--(6.5,1);
\draw [line width = 0.3mm, color = blue] (7.5,0)--(7.5,1);
\draw [line width = 0.3mm, color = blue] (8.5,0)--(8.5,1);
\draw [line width = 0.3mm, color = blue] (9.5,0)--(9.5,1);
\end{tikzpicture}
\end{center}
\begin{center}
\begin{tikzpicture}[scale = 1.25]
\draw [line width = 0.3mm, line cap = round] (0,0)--node[left]{0}(0,1);
\draw [line width = 0.3mm, line cap = round] (0,1)--node[above]{1}(1,1);
\draw [line width = 0.3mm, line cap = round] (1,1)--node[above]{2}(2,1);
\draw [line width = 0.3mm, line cap = round] (2,1)--node[above]{3}(3,1);
\draw [line width = 0.3mm, line cap = round] (3,1)--node[above]{4}(4,1);
\draw [line width = 0.3mm, line cap = round] (4,1)--node[above]{5}(5,1);
\draw [line width = 0.3mm, line cap = round] (5,0)--node[below]{3}(4,0);
\draw [line width = 0.3mm, line cap = round] (4,0)--node[below]{6}(3,0);
\draw [line width = 0.3mm, line cap = round] (3,0)--node[below]{4}(2,0);
\draw [line width = 0.3mm, line cap = round] (2,0)--node[below]{5}(1,0);
\draw [line width = 0.3mm, line cap = round] (1,0)--node[below]{2}(0,0);
\draw [line width = 0.3mm, line cap = round] (0+5,1)--node[above]{7}(1+5,1);
\draw [line width = 0.3mm, line cap = round] (1+5,1)--node[above]{6}(2+5,1);
\draw [line width = 0.3mm, line cap = round] (2+5,1)--node[above]{8}(3+5,1);
\draw [line width = 0.3mm, line cap = round] (3+5,1)--node[above]{9}(4+5,1);
\draw [line width = 0.3mm, line cap = round] (4+5,1)--node[above]{10}(5+5,1);
\draw [line width = 0.3mm, line cap = round] (5+5,1)--node[right]{0}(5+5,0);
\draw [line width = 0.3mm, line cap = round] (5+5,0)--node[below]{8}(4+5,0);
\draw [line width = 0.3mm, line cap = round] (4+5,0)--node[below]{1}(3+5,0);
\draw [line width = 0.3mm, line cap = round] (3+5,0)--node[below]{9}(2+5,0);
\draw [line width = 0.3mm, line cap = round] (2+5,0)--node[below]{7}(1+5,0);
\draw [line width = 0.3mm, line cap = round] (1+5,0)--node[below]{10}(0+5,0);
\draw [densely dashed](1,0)--(1,1);
\draw [densely dashed](2,0)--(2,1);
\draw [densely dashed](3,0)--(3,1);
\draw [densely dashed](4,0)--(4,1);
\draw [densely dashed](5,0)--(5,1);
\draw [densely dashed](1+5,0)--(1+5,1);
\draw [densely dashed](2+5,0)--(2+5,1);
\draw [densely dashed](3+5,0)--(3+5,1);
\draw [densely dashed](4+5,0)--(4+5,1);
\draw [fill = gray, draw opacity = 0, fill opacity = 0.25] (6,0) -- (7,0) -- (7,1) -- (6,1) -- cycle;
\draw [line width = 0.3mm, color = red] (0,0.5)--(10,0.5);
\draw [line width = 0.3mm, color = blue] (0.5,0)--(0.5,1);
\draw [line width = 0.3mm, color = blue] (1.5,0)--(1.5,1);
\draw [line width = 0.3mm, color = blue] (2.5,0)--(2.5,1);
\draw [line width = 0.3mm, color = blue] (3.5,0)--(3.5,1);
\draw [line width = 0.3mm, color = blue] (4.5,0)--(4.5,1);
\draw [line width = 0.3mm, color = blue] (5.5,0)--(5.5,1);
\draw [line width = 0.3mm, color = blue] (6.5,0)--(6.5,1);
\draw [line width = 0.3mm, color = blue] (7.5,0)--(7.5,1);
\draw [line width = 0.3mm, color = blue] (8.5,0)--(8.5,1);
\draw [line width = 0.3mm, color = blue] (9.5,0)--(9.5,1);
\end{tikzpicture}
\end{center}
\caption{Construction of a 1,1-square-tiled surface in $\calH(5,3)$.}\label{RSex}
\end{figure}
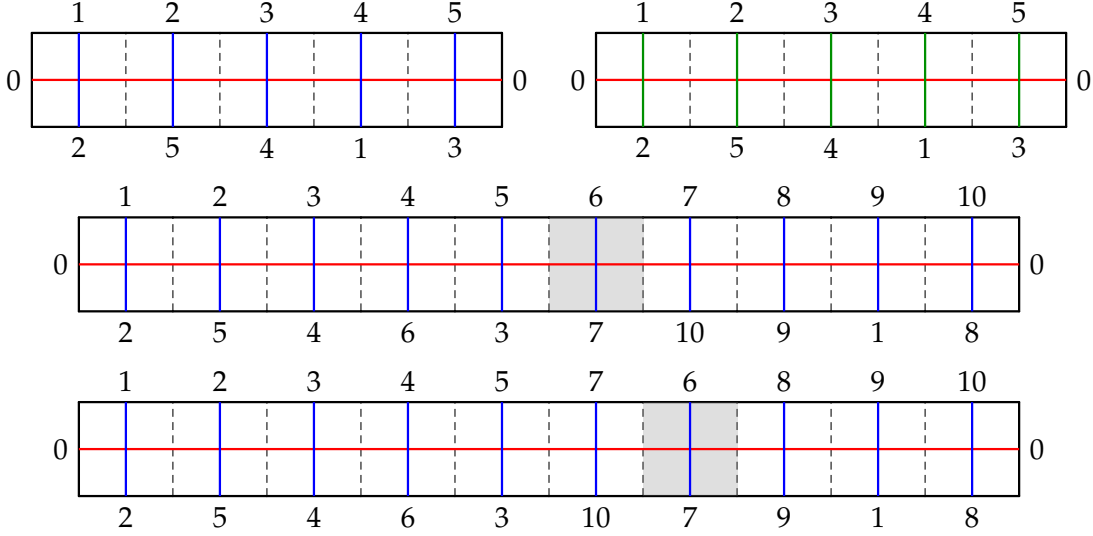
\end{center}

Alternatively, suppose that we start with the left-hand surface being the 1,1-square-tiled surface in $\odd(6)$ and the right-hand surface being the 1,1-square-tiled surface in $\odd(4)$ given by Proposition~\ref{Oddmin}. Now, after performing the cylinder concatenation, we will instead swap the shaded square with the square to its left. It can be checked that we obtain a 1,1-square-tiled surface in $\calH(7,3)$. Performing this operation with both surfaces being the 1,1-square-tiled surface in $\odd(6)$ results in a 1,1-square-tiled surface in $\calH(9,3)$. We shall call such a procedure a {\em left-swap concatenation}. The investigation of the outcomes of this procedure in general is the content of Proposition~\ref{LS}.

\begin{proposition}\label{RS}
For $j,k\geq 1$, let $S_{1}\in\odd(4j)$ and $S_{2}\in\odd(4k)$ be 1,1-square-tiled surfaces constructed from the permutation representatives given by Proposition~\ref{Oddmin}. Then the surface $S$ obtained by performing right-swap concatenation on these surfaces is a 1,1-square-tiled surface in $\calH(2(j+k)+1,2(j+k)-1)$. Recall that this stratum is nonempty and connected.

If instead we choose $S_{1}$ to be a 1,1-square-tiled surface in $\odd(4j+2)$ constructed from the permutation representative given by Proposition~\ref{Oddmin}, then the surface $S$ obtained by performing right-swap concatenation on these surfaces is a 1,1-square-tiled surface in $\non(2(j+k)+1,2(j+k)+1)$.

Moreover, the square-tiled surfaces produced have the minimum number of squares required for their respective strata.
\end{proposition}

\begin{proof}
We will view the process of right-swap concatenation from the point of view of filling pair diagrams. The filling pair diagram of a surface produced as in Lemma~\ref{comb1} is the end to end concatenation of the filling pair diagrams of the constituent surfaces where the edge that would have returned to the top of vertex 1 on the filling pair diagram of the first surface is connected to the top of what was vertex 1 on the second surface and vice versa. After this, the square swap corresponds to a vertex transposition of the vertices that were the first two vertices of the filling pair diagram of the second surface. We will keep track of the boundary components of the associated ribbon graph to determine the stratum of the resulting surface.

The boundary components of the ribbon graph associated to the filling pair diagram of the surface obtained by Lemma~\ref{comb1} are shown in Figure~\ref{rightswap1}. We read the diagram as follows. The two boundary components are represented by different line types. Following the orientation designated by the arrows, one counts the sides of the boundary components by starting at the outward arrow labelled by 1. We count this outgoing side. Then we continue to the next side labelled by 1 and of the same line type. If an incoming side has the same orientation (vertical or horizontal) as the outgoing side with the same label, then we do no count this incoming side, otherwise we do. We then add on the number of sides shown in brackets next to this incoming side. These numbers can be calculated by induction on the filling pair diagrams of Proposition~\ref{Oddmin}. We continue to count sides until we reach the next outgoing side and repeat as above. This continues until we return to where we started, that is, the outgoing side labelled by 1.

\begin{figure}[H]
\begin{center}
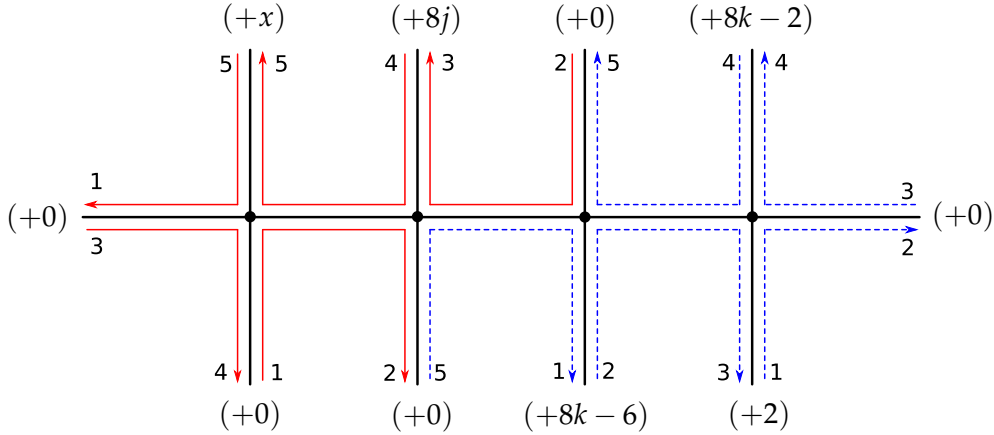
\end{center}
\caption{Boundary components of the ribbon graph before vertex transposition.}
\label{rightswap1}
\end{figure}

\begin{figure}[H]
\begin{center}
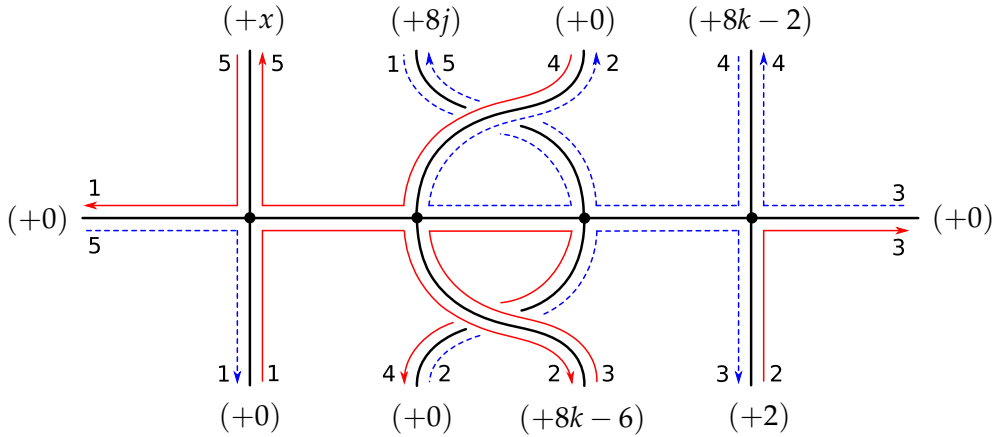
\end{center}
\caption{Boundary components of the ribbon graph after vertex transposition.}
\label{rightswap2}
\end{figure}

Note that in the diagram we have
\[x = \left\lbrace\begin{matrix}
8j-6,&\text{ for }\odd(4j),\\
8j+2, &\text{ for }\odd(4j+2).
\end{matrix}\color{white}\right\rbrace\]
It is then easy to see that the boundary components do give rise to zeros of the correct orders.

The effect of the vertex transposition on the boundary components of the ribbon graph associated to the filling pair diagram is shown in Figure~\ref{rightswap2}. We see that we have one boundary component with $8(j+k)+8$ sides corresponding to a zero of order $2(j+k)+1$, and a second boundary component with $8(j+k)$ sides corresponding to a zero of order $2(j+k)-1$, if $x = 8j-6$, or $8(j+k)+8$ sides corresponding to a zero of order $2(j+k)+1$, if $x = 8j+2$. As we have already shown that a 1,1-square-tiled surface in the hyperelliptic component of the stratum $\calH(g-1,g-1)$ requires strictly more than the minimum numbers of squares. It is clear that the 1,1-square-tiled surfaces we have produced representing $\calH(2(j+k)+1,2(j+k)+1)$ are in the nonhyperelliptic component since these surfaces have the minimum number of squares required for their respective strata which completes the proof of the proposition.
\end{proof}

Fixing $k=1$ in the above proposition, then choosing any $j\geq 1$ and applying the above construction using $\odd(4j)$ gives us 1,1-square-tiled surfaces in the strata $\calH(2j+3,2j+1)$, for $j\geq 1$. If instead we apply the above construction using $\odd(4j+2)$ then we produce 1,1-square-tiled surfaces in the nonhyperelliptic components of the strata $\calH(2j+3,2j+3)$, for $j\geq 1$. The following permutation, not produced by the above proposition, represents a 1,1-square-tiled surface in $\non(3,3)$
\[\left(\begin{matrix}
0&1&2&3&4&5&6&7&8 \\
2&8&6&5&7&4&1&3&0
\end{matrix}\right).\]

\begin{proposition}\label{LS}
For $j,k\geq 1$, let $S_{1}\in\odd(4j+2)$ and $S_{2}\in\odd(4k)$ be 1,1-square-tiled surfaces constructed from the permutation representatives given by Proposition~\ref{Oddmin}. Then the surface $S$ obtained by performing left-swap concatenation on these surfaces is a 1,1-square-tiled surface in $\calH(2(j+2k)+1,2j+1)$. Recall that this stratum is nonempty and connected.

If instead we choose $S_{2}$ to be a 1,1-square-tiled surface in $\odd(4k+2)$ constructed from the permutation representative given by Proposition~\ref{Oddmin}, then the surface $S$ obtained by performing left-swap concatenation on these surfaces is a 1,1-square-tiled surface in $\calH(2(j+2k)+3,2j+1)$. This stratum is also nonempty and connected.

Moreover, the square-tiled surfaces produced have the minimum number of squares required for their respective strata.
\end{proposition}

\begin{proof}
The proof is completely analogous to the proof of Proposition~\ref{RS}. The required filling pair diagram information is shown in Figures~\ref{leftswap1} and~\ref{leftswap2}. In this case
\[x = \left\lbrace\begin{matrix}
8k-2,&\text{ for }\odd(4k),\\
8k+6, &\text{ for }\odd(4k+2).
\end{matrix}\color{white}\right\rbrace\]
\end{proof}

\begin{figure}[H]
\begin{center}
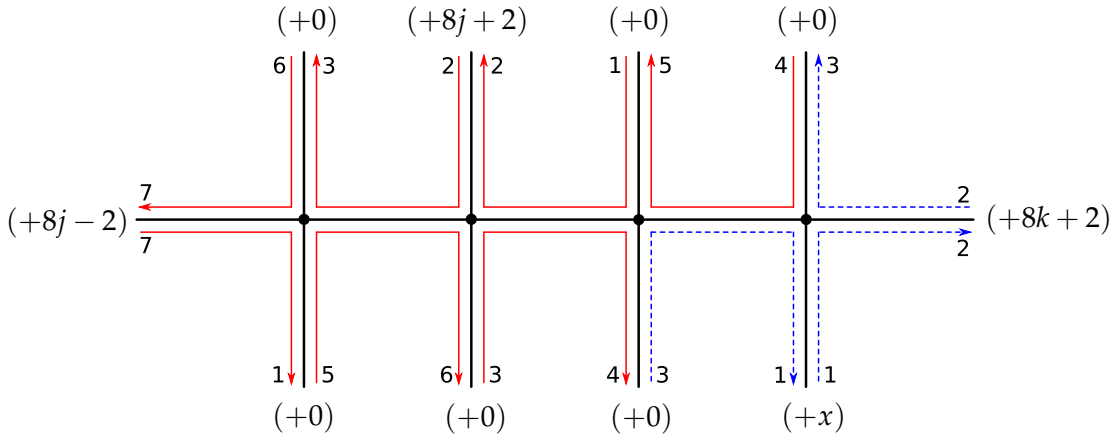
\end{center}
\caption{Boundary components of the ribbon graph before vertex transposition.}
\label{leftswap1}
\end{figure}

\begin{figure}[H]
\begin{center}
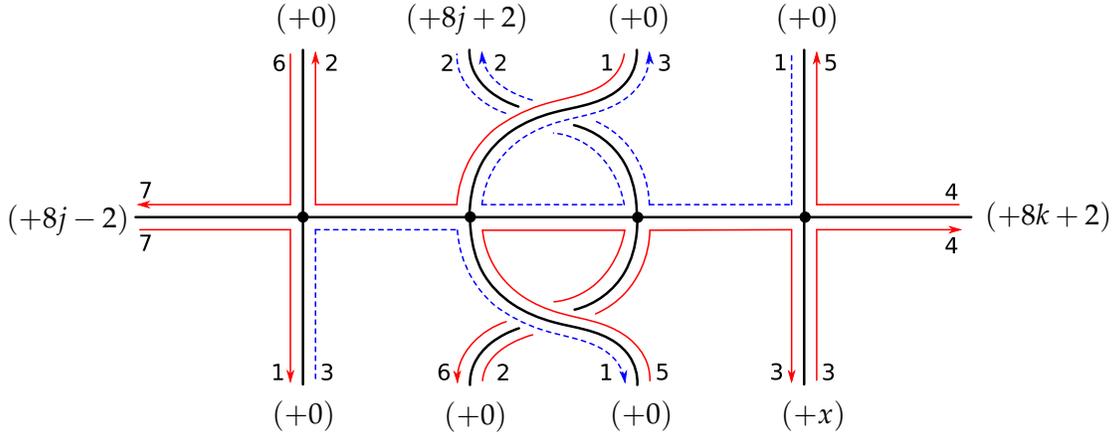
\end{center}
\caption{Boundary components of the ribbon graph after vertex transposition.}
\label{leftswap2}
\end{figure}

Observe that we have
\[2(j+2k)+1-(2j+1) = 4k\,\,\,\,\,\,\text{and}\,\,\,\,\,\,2(j+2k)+3-(2j+1) = 4k+2,\]
and so, since we have $j,k\geq 1$, the above proposition allows us to construct 1,1-square-tiled surfaces in the strata $\calH(2j+1+2n,2j+1)$, for $j\geq 1$ and $n\geq 2$.

We have yet to construct 1,1-square-tiled surfaces in strata with zeros of order 1. We first construct such surfaces in strata with a pair of odd order zeros, only one of which is a zero of order 1.

\begin{proposition}\label{2g-3,1}
The permutations
\begin{equation}\label{7,1}
\left(\begin{matrix}
0&1&2&3&4&5&6&7&8&9&10 \\
2&5&4&9&3&8&6&1&10&7&0
\end{matrix}\right)
\end{equation}
and, for $k\geq 4$,
\begin{equation}\label{2k+1,1}
\begin{multlined}
\left(\begin{matrix}
0&1&2&3&4&5&6&7&8&9&\cdots&2k-4&2k-3 \\
2&5&4&7&3&9&6&11&8&13&\cdots&2k-4&2k+3
\end{matrix}\color{white}\right) \\
\hspace{2cm}\color{white}\left(\color{black}\begin{matrix}
2k-2&2k-1&2k&2k+1&2k+2&2k+3&2k+4 \\
2k-2&2k+2&2k&1&2k+4&2k+1&0
\end{matrix}\right)
\end{multlined}
\end{equation}
represent 1,1-square-tiled surfaces in $\calH(7,1)$ and $\calH(2k+1,1)$, respectively. Moreover, these surfaces have the minimum number of squares necessary for their respective strata.
\end{proposition}

\begin{proof}
The proof is similar to the proof of Proposition~\ref{2Odd}. Indeed, we construct permutations (\ref{7,1}) and (\ref{2k+1,1}) by adding the 5 vertices in Figure~\ref{(7,1)} to the right-hand sides of the filling pair diagrams representing permutations (\ref{odd4}) and (\ref{oddmin}) of Proposition~\ref{Oddmin}. Here, we increase the order of the associated zero by 3 and add a separate zero of order 1.
\end{proof}

\begin{figure}[H]
\begin{center}
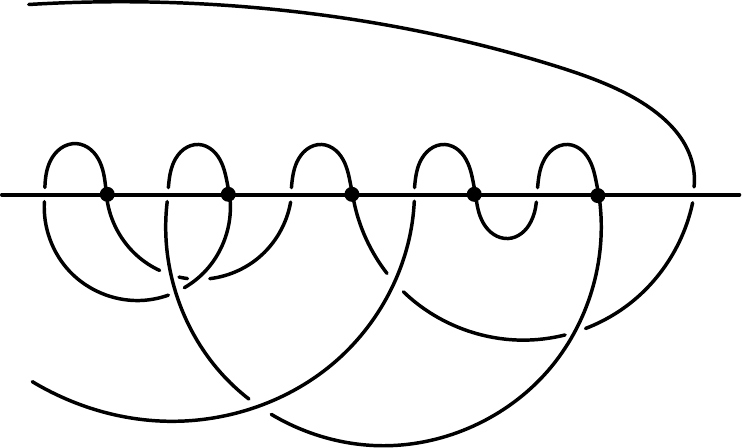
\end{center}
\caption{Filling pair diagram combinatorics used in Proposition~\ref{2g-3,1}.}
\label{(7,1)}
\end{figure}

The strata $\calH(3,1)$ and $\calH(5,1)$ are not covered by this proposition however the permutations
\[\left(\begin{matrix}
0&1&2&3&4&5&6 \\
2&5&1&6&4&3&0
\end{matrix}\right)\text{      and      }
\left(\begin{matrix}
0&1&2&3&4&5&6&7&8 \\
2&4&7&3&1&8&6&5&0
\end{matrix}\right)\]
represent 1,1-square-tiled surfaces in $\calH(3,1)$ and $\calH(5,1)$, respectively. \\

\paragraph*{{\bf Handling the hyperellipticity of $\bs{\calH(1,1)}$}}

As in the previous subsection, the hyperellipticity of genus two again causes us difficulty. In this case, we have no 1,1-square-tiled surface in $\calH(1,1)$ that we can use to build 1,1-square-tiled surfaces with the minimum number of squares. Observe that 1,1-square-tiled surfaces in the strata $\calH(1,1,1,1)$ and $\calH(1,1,1,1,1,1)$  are represented by the permutations
\[\left(\begin{matrix}
0&1&2&3&4&5&6&7&8 \\
2&6&5&3&1&8&4&7&0
\end{matrix}\right)\text{      and      }\left(\begin{matrix}
0&1&2&3&4&5&6&7&8&9&10&11&12 \\
2&8&1&5&11&7&3&10&6&12&9&4&0
\end{matrix}\right),\]
respectively.

We see then that the only strata with odd order zeros in which we cannot construct 1,1-square-tiled surfaces from those we have already constructed above are $\calH(2g-5,1,1,1)$, $g\geq 4$. The permutation
\[\left(\begin{matrix}
0&1&2&3&4&5&6&7&8&9&10 \\
2&10&6&5&1&8&4&7&3&9&0
\end{matrix}\right)\]
represents a 1,1-square-tiled surface in $\calH(3,1,1,1)$. The remaining cases can be constructed by using a method similar to that used to produce the 1,1-square-tiled surfaces in the proofs of Propositions~\ref{2Odd},~\ref{2Even}, and~\ref{2g-3,1}. Indeed, one can check that if we add the vertices shown in Figure~\ref{f:H11} to the right-hand side of the filling pair diagrams associated to the 1,1-square-tiled surfaces in $\calH(2k+1,1)$ constructed above, then we obtain 1,1-square-tiled surfaces in $\calH(2k+3,1,1,1)$. This now completes the construction of 1,1-square-tiled surfaces in all strata with zeros of odd order and hence the work of this subsection is complete.

\begin{figure}[H]
\begin{center}
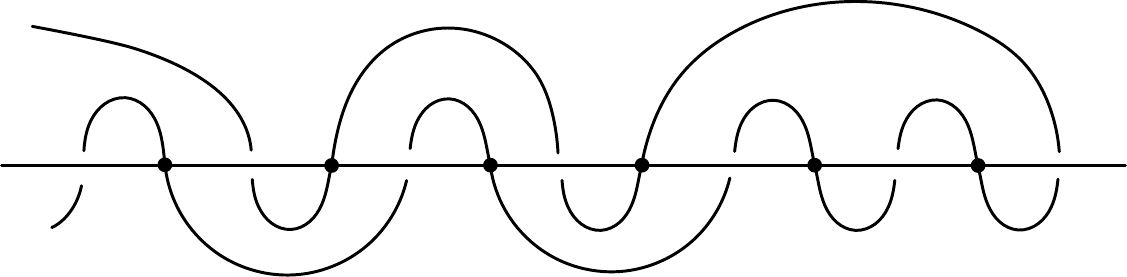
\end{center}
\caption{Filling pair diagram combinatorics for adding $\calH(1,1)$.}
\label{f:H11}
\end{figure}

\paragraph*{{\bf Inefficiency of Rauzy diagram search.}}

As mentioned in the introduction, the permutation representatives given in Proposition~\ref{2g-5} below, the proof of which we omit, are examples demonstrating the complexity of finding 1,1-square-tiled surfaces by using a shortest sequence of Rauzy moves on the permutation representatives given by Zorich. Indeed, these general forms were discovered by such a method and the sequences of Rauzy moves required, and hence the resulting permutation representatives, differed depending on the residue of $2g-5$ modulo 4. As such, we would not expect to be able to find a general method giving a sequence of Rauzy moves for more complicated strata.

\begin{proposition}\label{2g-5}
For $k\geq 1$, the permutation
\begin{equation}\label{4k+3}
\begin{multlined}
\hspace{-0.5cm}\left(\begin{matrix}
0&1&2&3&4&5&6&7&8&\cdots&4k+3 \\
2&4k+10&4k+6&4k+5&1&4&3&6&5&\cdots&4k+2
\end{matrix}\color{white}\right) \\
\\
\hspace{1.2cm}\color{white}\left(\color{black}\begin{matrix}
4k+4&4k+5&4k+6&4k+7&4k+8&4k+9&4k+10 \\
4k+1&4k+8&4k+4&4k+7&4k+3&4k+9&0
\end{matrix}\right)
\end{multlined}
\end{equation}
represents a 1,1-square-tiled surface in the stratum $\calH(4k+3,1,1,1)$. For $k\geq 2$, the permutation
\begin{equation}\label{4k+1}
\begin{multlined}
\hspace{-0.5cm}\left(\begin{matrix}
0&1&2&3&4&5&6&7&8&\cdots&4k-1&4k \\
2&4k+8&4k+5&4k+4&1&4&3&6&5&\cdots&4k-2&4k-3
\end{matrix}\color{white}\right) \\
\\
\hspace{2cm}\color{white}\left(\color{black}\begin{matrix}
4k+1&4k+2&4k+3&4k+4&4k+5&4k+6&4k+7&4k+8 \\
4k+3&4k-1&4k+2&4k+6&4k+1&4k&4k+7&0
\end{matrix}\right)
\end{multlined}
\end{equation}
represents a 1,1-square-tiled surface in $\calH(4k+1,1,1,1)$.
All such surfaces achieve the minimal number of squares for their respective strata.
\end{proposition}

\subsection{General strata}\label{Gen}

In this subsection, we complete the proof of Theorem~\ref{STS} by constructing 1,1-square-tiled surfaces in general strata; that is, we construct 1,1-square-tiled surfaces in strata with both even and odd order zeros. Recall that such strata are connected.

One can check that the only strata in which we are not already able to construct 1,1-square-tiled surfaces are those whose only even order zero is a zero of order 2, and those whose only odd order zeros are a pair of zeros of order 1.

In the former case, if we can construct 1,1-square-tiled surfaces in all strata with two odd order zeros and a zero of order 2 then we can use the surfaces we constructed in the previous subsection to complete this case. We have one exception, that being $\calH(2,1,1,1,1)$ however a 1,1-square-tiled surface in this stratum is represented by the permutation
\[\left(\begin{matrix}
0&1&2&3&4&5&6&7&8&9&10&11 \\
2&7&11&6&3&9&5&1&8&4&10&0
\end{matrix}\right).\]
We will revisit the technique we used in Propositions~\ref{2Odd} and~\ref{2Even}. That is, the addition of the 5 vertices in Figure~\ref{f:H2} to the right-hand side of the filling pair diagram to produce the zero of order 2. However, we must proceed with more care than we did in the proof of Proposition~\ref{2Odd} as one can check that when we add these vertices to the filling pair diagram of a surface with two zeros it is only the zero corresponding to the boundary component of the ribbon graph that `leaves' the filling pair diagram on the right below the horizontal line (see Figure~\ref{addbot1}) that has its order increased. This can easily be seen by observing the combinatorics of the filling pair diagrams. 

\begin{figure}[H]
\begin{center}
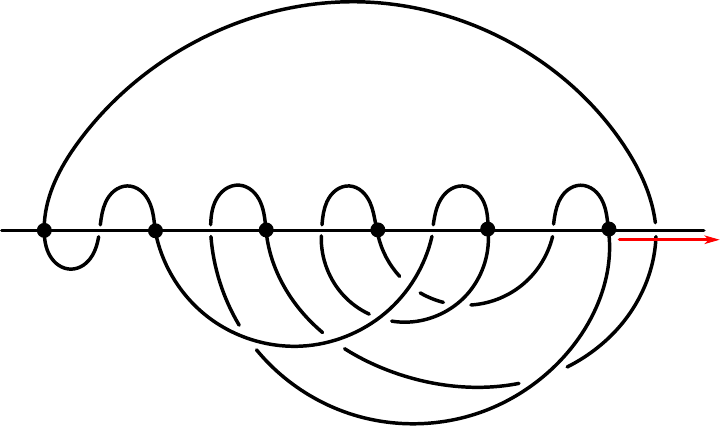
\end{center}
\caption[The zero of order 3 leaves the filling pair diagram at the bottom.]{The boundary component corresponding to the zero of order 3 leaves the filling pair diagram at the bottom.}
\label{addbot1}
\end{figure}

With this in mind, we need only keep track of which zero is associated to the boundary component that leaves on the bottom for the surfaces that we constructed in the previous subsection. It is easy to check that for the 1,1-square-tiled surfaces in the strata $\calH(2k+1,2k+1+2n)$, for $k\geq 1$ and $n\geq 2$, constructed in Proposition~\ref{LS}, that the boundary component that leaves on the bottom is the one associated to the zero of order $2k+1+2n$. Hence, we can construct 1,1-square-tiled surfaces in the strata $\calH(2k+1,2k+1+2n,2)$, for $k\geq 1$ and $n\geq 3$. Using the 1,1-square-tiled surfaces in the strata $\calH(2k+1,2k+1)$, $k\geq 1$, constructed in (and after for $\calH(3,3)$) Proposition~\ref{RS}, we can construct 1,1-square-tiled surfaces in the strata $\calH(2k+3,2k+1,2)$, $k\geq 1$. Moreover, the 1,1-square-tiled surfaces in the strata $\calH(2k+1,1)$, $k\geq 1$, constructed in and after Proposition~\ref{2g-3,1}, have the boundary component that leaves on the bottom being the one associated to the zero of order $2k+1$, and so we can construct 1,1-square-tiled surfaces in the strata $\calH(2k+3,1,2)$, $k\geq 1$. Finally, we observe that for the 1,1-square-tiled surfaces in the strata $\calH(2k+3,2k+1)$, $k\geq 1$, constructed in Proposition~\ref{RS}, the boundary component that leaves on the bottom is the one associated to the zero of order $2k+3$. Hence, we can construct 1,1-square-tiled surfaces in the strata $\calH(2k+5,2k+1,2)$, $k\geq 1$.

The only strata not covered thus far are $\calH(3,1,2)$ and $\calH(2k+1,2k+1,2)$, $k\geq 1$. The permutation
\[\left(\begin{matrix}
0&1&2&3&4&5&6&7&8&9 \\
2&6&8&3&7&4&1&9&5&0
\end{matrix}\right)\]
represents a 1,1-square-tiled surface in the stratum $\calH(3,1,2)$. Hence, we see that to complete the proof of Theorem~\ref{STS} we must construct 1,1-square-tiled surfaces in the strata $\calH(2k+3,2k+1)$, $k\geq 0$, with the boundary component that leaves on the bottom being the one associated to the zero of order $2k+1$. Indeed, we would then be able to construct 1,1-square-tiled surfaces in the strata $\calH(2k+3,2k+3,2)$, $k\geq 0$. This is completed by the following proposition and the permutation representative that follows.

\begin{proposition}\label{2k+3,2k+1bot}
The permutations
\begin{equation}\label{bot1}
\left(\begin{matrix}
0&1&2&3&4&5&6&7&8&9&10 \\
2&6&4&10&8&3&1&9&7&5&0
\end{matrix}\right),
\end{equation}
and, for $k\geq 2$,
\begin{equation}\label{bot2}
\begin{multlined}
\left(\begin{matrix}
0&1&2&3&4&5&6&7&8&9&10&11 \\
2&6&4&10&8&3&12&9&7&5&14&11
\end{matrix}\color{white}\right) \\
\hspace{2.5cm}\color{white}\left(\color{black}\begin{matrix}
12&13&14&\cdots&4k+3&4k+4&4k+5&4k+6 \\
16&13&18&\cdots&4k+3&1&4k+5&0
\end{matrix}\right)
\end{multlined}
\end{equation}
represent 1,1-square-tiled surfaces in $\calH(5,3)$ and $\calH(2k+3,2k+1)$, respectively, with the boundary component that leaves the filling pair diagram on the bottom being associated to the zeros of order 3 and $2k+1$, respectively. Moreover, these surfaces have the minimum number of squares necessary for their respective strata.
\end{proposition}

\begin{proof}
It is easy to check that these permutation representatives do indeed represent 1,1-square-tiled surfaces in the claimed strata. Further, it is simple to check that the associated filling pair diagrams have the claimed combinatorics.
\end{proof}

A 1,1-square-tiled surface in the stratum $\calH(3,1)$ with the boundary component that leaves on the bottom being the one associated to the zero of order 1 is represented by the permutation
\[\left(\begin{matrix}
0&1&2&3&4&5&6 \\
2&6&5&1&4&3&0
\end{matrix}\right).\]

We now construct 1,1-square-tiled surfaces in the strata $\calH(2k,1,1)$, $k\geq 1$. We can then use surfaces we have already constructed to complete the remaining cases apart from the stratum $\calH(2,2,1,1)$ which is represented by the permutation
\[\left(\begin{matrix}
0&1&2&3&4&5&6&7&8&9&10 \\
2&4&9&7&3&8&5&1&10&6&0
\end{matrix}\right).\]
We note that the permutations
\[\left(\begin{matrix}
0&1&2&3&4&5&6&7 \\
2&6&4&1&7&5&3&0
\end{matrix}\right),\]
and
\[\left(\begin{matrix}
0&1&2&3&4&5&6&7&8&9 \\
2&7&4&1&9&5&8&6&3&0
\end{matrix}\right)\]
represent 1,1-square-tiled surfaces in $\calH(2,1,1)$ and $\calH(4,1,1)$, respectively. We can then construct 1,1-square-tiled surfaces in $\calH(2k,1,1)$, for $k\geq 3$, by adding the combinatorics of Figure~\ref{f:H11} to the filling pair diagrams associated to the 1,1-square-tiled surfaces in $\odd(2k)$ that we produced in Proposition~\ref{Oddmin}.

This completes the proof of Theorem~\ref{STS}. \\

\paragraph*{{\bf Inefficiency of Rauzy diagram search revisited.}}

Further to Proposition~\ref{2g-5}, the following proposition is a further example of the inefficiency of a Rauzy diagram search to find 1,1-square-tiled surfaces. In this case, which is less complicated than that considered in Proposition~\ref{2g-5}, we again find differing sequences of Rauzy moves, and hence differing permutation representatives, depending on the residue of the even order modulo 4. We again omit the proof.

\begin{proposition}
For $k\geq 1$, the permutation
\begin{equation}\label{4k+2,1,1}
\begin{multlined}
\hspace{-1cm}\left(\begin{matrix}
0&1&2&3&4&5&6&7&\cdots&4k+2 \\
2&6&4&1&8&7&10&9&\cdots&4k+6
\end{matrix}\color{white}\right) \\
\\
\hspace{1.5cm}\color{white}\left(\color{black}\begin{matrix}
4k+3&4k+4&4k+5&4k+6&4k+7 \\
4k+5&4k+7&5&3&0
\end{matrix}\right)
\end{multlined}
\end{equation}
represents a 1,1-square-tiled surface in $\calH(4k+2,1,1)$. For $k\geq 2$, the permutation
\begin{equation}\label{4k,1,1}
\begin{multlined}
\hspace{-1cm}\left(\begin{matrix}
0&1&2&3&4&5&6&7&\cdots&4k-2&4k-1 \\
2&7&4&1&9&8&11&10&\cdots&4k+3&4k+2
\end{matrix}\color{white}\right) \\
\\
\hspace{3cm}\color{white}\left(\color{black}\begin{matrix}
4k&4k+1&4k+2&4k+3&4k+4&4k+5 \\
4k+5&5&4k+4&6&3&0
\end{matrix}\right)
\end{multlined}
\end{equation}
represents a 1,1-square-tiled surfaces in the stratum $\calH(4k,1,1)$. Moreover, these surfaces have the minimum number of squares necessary for their respective strata.
\end{proposition}


\section{Ratio-optimising pseudo-Anosovs}\label{RO1}

This section contains the proof of Theorem~\ref{RO}.

\subsection{\teichmuller preliminaries} Recall that the {\em \teichmuller space}, $\T(S)$, of a closed surface $S$ of genus $g\geq 2$ is the set of equivalence classes of pairs $(X,f)$ where $X$ is a Riemann surface of genus $g$ and $f:S\to X$, called a marking, is a homeomorphism. Two such pairs $(X,f)$ and $(Y,g)$ are equivalent if there exists a conformal map $h:X\to Y$ such that $h\circ f$ is isotopic to $g$. We will abuse notation and denote $[(X,f)]$ by $X$. By the uniformisation theorem, a point $X\in\T(S)$ determines a hyperbolic metric on $S$ up to isometries isotopic to the identity. As such, given an isotopy class $[\alpha]$ of an essential simple closed curve $\alpha$ on the surface $S$ we can talk about its length in the hyperbolic metric determined by the point $X$. The \teichmuller space carries a metric $d_{\T}$, called the {\em \teichmuller metric} and from now on we will denote by $\T(S)$ the metric space $(\T(S),d_{\T})$. Given a pseudo-Anosov homeomorphism $f$, we define the {\em translation length} of $f$ on $\T(S)$ to be $\ell_{\T}(f):=\log(\lambda_{f})$, where $\lambda_{f}$ is the dilatation of $f$.

Given an Abelian differential on the surface $S$, the translation structure induces a complex structure on $S$ which, taking the identity homeomorphism as the marking, determines a point in \teichmuller space. The group $\SL(2,\R)$ acts on the space of Abelian differentials by its natural action on the polygons in $\Cbb$ given by the translation structure on $S$. The action of $\SO(2,\R)$ on an Abelian differential does not change the point in $\T(S)$ that it determines. As such, the orbit of an Abelian differential under the action of $\SO(2,\R)\!\setminus\!\SL(2,\R)$ gives an embedding of $\SO(2,\R)\!\setminus\!\SL(2,\R)\cong\Hbb$ into $\T(S)$. The image of this embedding is called the {\em \teichmuller disk} of the Abelian differential.

The {\em curve graph}, $\C(S)$, of the surface $S$ is the 1-skeleton of the curve complex introduced by Harvey \cite{H}. The vertices are isotopy classes of essential simple closed curves on the surface $S$, with two vertices joined by an edge if and only if they can be realised disjointly on $S$. We will abuse notation and denote $[\alpha]\in\C(S)$ by $\alpha$. Assigning length 1 to each edge, we equip $\C(S)$ with the associated path metric $d_{\C}$. We will denote by $\C(S)$ the metric space $(\C(S),d_{\C})$. Given a pseudo-Anosov homeomorphism $f$, we define the {\em asymptotic translation length} of $f$ on $\C(S)$ to be
\[\ell_{\C}(f):=\liminf_{n\to \infty}\frac{d_{\C}(f^{n}(\alpha),\alpha)}{n},\]
for any $\alpha\in\C(S)$. We claim that this is well-defined. Indeed, by work of Bowditch~\cite{Bow} and Masur-Minsky~\cite{MM}, for a pseudo-Anosov homeomorphism this limit inferior is in fact a strictly positive limit. The independence of the choice of $\alpha$ then follows by taking the appropriate limits on the inequalities
\[d_{\C}(f^{n}(\alpha),\alpha)-2d_{\C}(\alpha,\beta)\leq d_{\C}(f^{n}(\beta),\beta)\leq d_{\C}(f^{n}(\alpha),\alpha)+2d_{\C}(\alpha,\beta),\]
which are obtained from the triangle inequality and the fact that the mapping class group acts by isometries on the curve graph.

\subsection{The systole map and Lipschitz constant} We now define the {\em systole map}, $\sys: \T(S)\to \C(S)$, to be the coarsely-defined map that sends a point $X\in\T(S)$ to the isotopy class of the curve with shortest length in the hyperbolic metric determined by $X$, known as the systole. The map is only coarsely-defined as there can be multiple systoles on a surface, however the set of systoles on $X$ is a set of diameter at most 2 in $\C(S)$. Indeed, on a compact surface one has that $i(\alpha,\beta)\leq 1$ for any two systoles $\alpha$ and $\beta$. So in particular the curves $\alpha$ and $\beta$ cannot form a filling pair from which it follows that $d_{\C}(\alpha,\beta)\leq 2$. We will abuse notation and think of $\sys$ as a well-defined map. The study of this map played a key role in the work of Masur and Minsky in which they proved that the curve complex is $\delta$-hyperbolic \cite{MM}. They showed in particular that the map is coarsely $K$-Lipschitz. That is, there exists a $C\geq 0$ such that
\[d_{\C}(\sys(X),\sys(Y))\leq K\cdot d_{\T}(X,Y)+C,\]
for all $X,Y\in\T(S)$.

It is natural to ask what is the optimum Lipschitz constant, $\kappa_{g}$, defined by
\[\kappa_{g}:=\inf\{K>0\,|\,\exists\,C\geq 0\,\text{such that sys is coarsely K-Lipschitz}\},\]
and Gadre-Hironaka-Kent-Leininger determined that the ratio of $\kappa_{g}$ to $1/\log(g)$ is bounded from above and below by two positive constants {\cite[Theorem 1.1]{GHKL}}. In such a case, we use the notation $\kappa_{g}\asymp 1/\log(g)$, and say that $\kappa_{g}$ is comparable to $1/\log(g)$. To find an upper bound for $\kappa_{g}$, Gadre-Hironaka-Kent-Leininger gave a careful version of the proof of Masur-Minsky that $\sys$ is coarsely Lipschitz. They then constructed pseudo-Anosov homeomorphisms for which the ratio $\ell_{\C}(f)/\ell_{\T}(f)\asymp 1/\log(g)$, where $\ell_{\C}(f)$ and $\ell_{\T}(f)$ are the asymptotic translation lengths of $f$ in $\C(S)$ and $\T(S)$, respectively. They then obtained a lower bound for $\kappa_{g}$ by noting that, for any pseudo-Anosov homeomorphism $f$, we have
\[\kappa_{g}\geq\frac{\ell_{\C}(f)}{\ell_{\T}(f)}.\]

\subsection{Constructing ratio-optimising pseudo-Anosov homeomorphisms} Recall that a pair of essential simple closed curves which are in minimal position on a surface $S$ are said to be a {\em filling pair} if the complement of their union is a disjoint collection of disks. Using a Thurston construction on filling pairs, Aougab-Taylor constructed a larger family of pseudo-Anosov homeomorphisms for which $\tau(f) := \ell_{\T}(f)/\ell_{\C}(f)$ was bounded above by a function $F(g)\asymp\log(g)$ {\cite[Theorem 1.1]{AT2}}. Such homeomorphisms are said to be {\em ratio-optimising}. Moreover, they proved that there exists a \teichmuller disk $\D\simeq\Hbb\subset\T(S)$ such that there exist infinitely many conjugacy classes of primitive ratio-optimising pseudo-Anosovs $f$ with the invariant axis of $f$ being contained in $\D$.

To construct these pseudo-Anosovs, Aougab-Taylor began with a pair of simple closed curves $\alpha$ and $\beta$ that filled the surface $S$. They then took high powers, independent of the genus of the surface, of the Dehn twists about each curve and showed that the Bass-Serre tree of the free group generated by these elements quasi-isometrically embeds in $\C(S)$. This then allowed them to bound the asymptotic translation length, $\ell_{\C}$, of elements of this group in terms of their syllable length. They also bounded $\ell_{\T}$ in terms of the syllable length of the element and the geometric intersection number, $i(\alpha,\beta)$, of the filling pair. From this they were able to deduce that, for pseudo-Anosov elements of this free group,
\[\tau(f)\leq\log(D\cdot i(\alpha,\beta)),\]
where $D$ is a constant independent of the genus of $S$. Ratio-optimising pseudo-Anosovs were then constructed by using filling pairs for which $i(\alpha,\beta)\asymp g$.

Recall that a filling pair $(\alpha,\beta)$ on a surface $S$, with all intersections occurring with the same orientation, determines an Abelian differential on that surface. We will denote the \teichmuller disk of this Abelian differential by $\D(\alpha,\beta)$. The ratio-optimising pseudo-Anosovs produced from this filling pair will stabilise $\D(\alpha,\beta)$ and, moreover, their invariant axis will be contained in $\D(\alpha,\beta)$. Aougab-Taylor used the hyperbolicity of $\C(S)$ and the acylindricity of the action of $\Mod(S)$ on $\C(S)$ to show that in fact there are infinitely many conjugacy classes of primitive ratio-optimising pseudo-Anosovs constructed from this filling pair that have this property. We remark that their theorem deals with the general case of filling pairs that determine quadratic differentials on a punctured surface $S_{g,p}$. We are specialising to the case of Abelian differentials on closed surfaces.

\subsection{Proof of Theorem~\ref{RO}} Fix $g$ and let $\mathscr{C}$ be any connected component of any stratum of $\calH$. By Theorem~\ref{STS}, we can find a 1,1-square-tiled surface in $\mathscr{C}$. The core curves, $\alpha$ and $\beta$, of the vertical and horizontal cylinders of this surface form a filling pair and so we can construct pseudo-Anosovs from this filling pair using the above technique of Aougab-Taylor. For any such pseudo-Anosov, we have
\[\tau(f)\leq\log(D\cdot i(\alpha,\beta))\leq \log(D\cdot (4g-2))\asymp \log(g),\]
since the greatest number of squares required for a 1,1-square-tiled surface of genus $g$, and so the greatest intersection number of the associated filling pair, is given by the connected component $\hyp(g-1,g-1)$ which requires $4g-2$ squares. Hence we have that the pseudo-Anosovs are ratio-optimising. Moreover, as above, we have infinitely many conjugacy classes of primitive ratio-optimising pseudo-Anosovs having their invariant axis contained in the \teichmuller disk determined by this 1,1-square-tiled surface. As such, we have completed the proof of Theorem~\ref{RO}.

Note that this extends the abundance result of Aougab-Taylor. That is, not only are there infinitely many conjugacy classes of primitive ratio-optimising pseudo-Anosovs in a \teichmuller disk of $\T(S)$ but this \teichmuller disk can be taken to be the \teichmuller disk of an Abelian differential from any connected component of any stratum of $\calH$.


\section{Filling pairs on punctured surfaces}\label{FP1}

Here we will prove Theorem~\ref{FP}.

Let $S_{g,p}$ denote the surface of genus $g\geq{0}$ with $p\geq{0}$ punctures. We define $i_{g,p}$ to be the minimal geometric intersection number for a filling pair on $S_{g,p}$. The values of $i_{g,p}$ were determined in the works of Aougab-Huang \cite{AH}, Aougab-Taylor \cite{AT1}, and the author \cite{J}, and can be summarised as follows.

\begin{theorem}
The values of $i_{g,p}$ are the following:
\begin{itemize}
\item[(1)] If $g\neq 2,0$ and $p=0$, then $i_{g,p}=2g-1$; 
\item[(2)] If $g\neq 2,0$ and $p\geq1$, then $i_{g,p}=2g+p-2$; 
\item[(3)] If $g=0$ and $p\geq4$, then $i_{g,p}=p-2$ if $p$ is even, and $i_{g,p}=p-1$ if $p$ is odd; 
\item[(4)] If $g=2$ and $p\leq 2$, then $i_{g,p}=4$; 
\item[(5)] If $g=2$ and $p\geq 2$, then $i_{g,p}=2g+p-2$.
\end{itemize}
\end{theorem}

One can ask, for $g\geq 1$, whether $i_{g,p}$ can be realised as the algebraic intersection number, $\widehat{i}(\alpha,\beta)$, of a filling pair $(\alpha,\beta)$. Aougab-Menasco-Nieland \cite{AMN} answered this question for the case of $i_{g,0}$; that is, for minimally intersecting filling pairs on closed surfaces. Moreover, they were interested in counting the number of mapping class group orbits of such filling pairs. Their method involves algebraically constructing 1,1-square-tiled surfaces with the minimum number of squares in the stratum $\calH(2g-2)$, which they call square-tiled surfaces with connected leaves. The core curves of the cylinders of such surfaces give rise to filling pairs with algebraic intersection number equal to $i_{g,0}$.

For $n\geq{i_{g,p}}$, by a {\em compatible decomposition} of the surface $S_{g,p}$ into $n+2-2g$ many $4k$-gons, we mean a decomposition of the surface into $4k$-gons $P_{1},\ldots,P_{n+2-2g}$ such that, if $P_{i}$ is a $4k_{i}$-gon, then $\sum (k_{i}-1) = 2g-2$.

Observe that a filling pair on the surface $S_{g,p}$ with $\widehat{i}(\alpha,\beta) = i(\alpha,\beta) = n\geq i_{g,p}$ divides the surface into a collection of $n+2-2g$ many $4k$-gons forming a compatible decomposition. Conversely, given an appropriate choice of orientation, the core curves of a 1,1-square-tiled surface with $n$ squares and $n+2-2g$ many zeros, of orders greater than or equal to zero, form a filling pair with $\widehat{i}(\alpha,\beta) = i(\alpha,\beta) = n$ dividing the surface into $n+2-2g$ many $4k$-gons with a zero of order $k-1$ giving rise to a $4k$-gon. These $4k$-gons also form a compatible decomposition.

Note that the square torus can be represented by the permutation
\[\left(\begin{matrix}
0&1 \\
1&0
\end{matrix}\right)\]
and that this permutation can be combined with a 1,1-square-tiled surface of genus $g$, by cylinder concatenation as in Lemma~\ref{comb1}, to produce another 1,1-square-tiled surface of genus $g$. This process will add a zero of order 0 to the surface and one additional square.

For $g\geq3$ and $p\geq 0$, let $n\geq i_{g,p}$ and choose a compatible decomposition of the surface $S_{g,p}$ into $n+2-2g$ many $4k$-gons, as described above. There will be a number, less than or equal to $2g-2$, of these $4k$-gons having $k\geq 2$. Let $k_{1},\ldots,k_{m}$ be the list of these $k$ values. By Theorem~\ref{STS}, we can choose a 1,1-square-tiled surface in the stratum $\calH(k_{1}-1,\ldots,k_{m}-1)$ with $2g+m-2$ squares. Adding $(n+2-2g-m)$ zeros of order 0 to this surface, using the method described in the previous paragraph, and choosing orientations appropriately, we will have a 1,1-square-tiled surface such that the core curves of the cylinders form a filling pair $(\alpha,\beta)$ with $\widehat{i}(\alpha,\beta) = i(\alpha,\beta) = n$ which, after adding $p$ punctures  to distinct complementary regions of the filling pair, gives rise to the specified polygonal decomposition of $S_{g,p}$.

For $g=2$, we cannot add zeros of order 0 as above since the permutation representatives for $\hyp(2)$ and $\hyp(1,1)$ do not have the correct form. However, we can split symbol $3$ in the permutations for these components given by Proposition~\ref{Hyp} to add zeros of order 0, before adding punctures to the complementary regions. In the case $g=1$, we need only combine, as in Lemma~\ref{comb1}, the permutation for the square torus above with itself $n$ times and then add $p$ punctures to the surface in distinct complementary regions of the filling pair. This completes the proof of Theorem~\ref{FP}.


\end{document}